\documentclass[a4paper,11pt]{article}
\usepackage{graphicx} 
\usepackage{amsmath}
\usepackage{amssymb}
\usepackage{amsfonts}
\usepackage{amsthm}
\usepackage{parskip}
\usepackage{comment}
\usepackage{xcolor}
\usepackage{tikz}
\usepackage{yhmath}
\usepackage{enumitem}
\usepackage[top=22mm, bottom=23mm, left=22mm, right=22mm]{geometry}
\allowdisplaybreaks

\newtheorem{thm}{Theorem}[section]
\newtheorem{prop}[thm]{Proposition}
\newtheorem{lem}[thm]{Lemma}
\newtheorem{cor}[thm]{Corollary}
\newtheorem{conj}[thm]{Conjecture}

\theoremstyle{definition}

\newtheorem{df}[thm]{\bf Definition}

\newcommand{\PP}[1]{\mathbb{P}\left(#1\right)}
\newcommand{\floor}[1]{\left\lfloor#1\right\rfloor}
\newcommand{\ceil}[1]{\left\lceil#1\right\rceil}

\newcommand{\EE}[1]{\mathbb{E}\left(#1\right)}

\title{Multivariate and quantitative Erd\H{o}s--Kac laws for Beatty sequences}
\author{Fredy Yip\thanks{Trinity College, University of Cambridge, United Kingdom. Email: \textbf{fy276@cam.ac.uk}.}}
\date{}
\begin{document}

\maketitle

\begin{abstract}
    The classical Erd\H{o}s--Kac theorem states that for $n$ chosen uniformly at random from $1, \dots, N$, the random variable $(\omega(n) - \log\log N)/\sqrt{\log\log N}$ converges in distribution to the standard Gaussian as $N$ tends to infinity. Banks and Shparlinski showed that this Gaussian convergence holds for any Beatty sequence $\floor{\alpha n + \beta}$ in place of $n$. Continuing in this spirit, Crn{\v{c}}evi{\'c}, Hern{\'a}ndez, Rizk, Sereesuchart and Tao considered the joint distribution of $\omega(n)$ and $\omega(\floor{\alpha n})$, which they showed to be asymptotically independent for irrational values of $\alpha$. 

    Generalising both results, we show that for any positive integer $k$, real numbers $\alpha_1, \dots, \alpha_k > 0$ and $\beta_1, \dots, \beta_k$, where $\alpha_i/\alpha_j$ is irrational for $i\neq j$, the joint distribution of $(\omega(\floor{\alpha_in + \beta_i}) - \log\log N)/\sqrt{\log\log N}$ converges to the $k$-dimensional standard Gaussian. 

    We next discuss quantitative bounds on the rate of convergence which do not depend on the values taken by the relevant parameters. Banks and Shparlinski remarked that such quantitative bounds may be given for a single Beatty sequence $\floor{\alpha n + \beta}$ under Diophantine type assumptions on $\alpha$. We show that such assumptions are in fact unnecessary. Specifically, for any real numbers $\alpha > 0, \beta$, we show that the Kolmogorov distance between the random variable $(\omega(\floor{\alpha n + \beta}) - \log\log N)/\sqrt{\log\log N}$ and the standard Gaussian is bounded above by $O(\log\log\log N/\sqrt{\log\log N})$ as $N$ tends to infinity. 

    On the other hand, we show that universal quantitative bounds of this kind do not exist for higher-degree generalised polynomials or for the joint convergence for multiple Beatty sequences. 
\end{abstract}

\section{Introduction}

In their seminal work in probabilistic number theory, Erd\H{o}s and Kac~\cite{E-K} established the distribution of the number $\omega(n)$ of distinct prime divisors of a positive integer $n$ chosen uniformly from $[N] = \{1, \dots, N\}$. Specifically, they showed that $\omega(n)$ converges in distribution to a normal random variable centred at $\log\log N$ with variance $\log\log N$. 

\begin{thm}[Erd\H{o}s and Kac~\cite{E-K}] \label{E-K thm}
    For $n\sim U[N]$ chosen uniformly at random from $[N]$, the random variable 
    \begin{equation*}
        \frac{\omega(n) - \log\log N}{\sqrt{\log\log N}}
    \end{equation*}
    converges in distribution to the standard Gaussian $\mathcal{N}(0, 1)$ as $N\rightarrow\infty$. 
\end{thm}

Many alternative proofs of the Erd\H{o}s--Kac theorem (Theorem~\ref{E-K thm} above) have been discovered. Most notably, Halberstam~\cite{method-of-moments} proved the Erd\H{o}s--Kac theorem by showing that all moments of the random variable $(\omega(n) - \log\log N)/\sqrt{\log\log N}$ converged to the corresponding moment of the standard Gaussian. Granville and Soundararajan~\cite{method-of-moments-simplified} gave a simplified account of this proof of the Erd\H{o}s--Kac theorem via the ``method of moments". 

The Erd\H{o}s--Kac theorem admits numerous generalisations over a diverse range of settings, including the distribution of $\omega(f(n))$ for an irreducible integer polynomial $f$~\cite{polynomials}, and $\omega(\floor{\alpha n + \beta})$ for a Beatty sequence $\floor{\alpha n + \beta}$~\cite{B-S}. 

\begin{thm}[Halberstam~\cite{polynomials}, Theorem 3] \label{integer polynomial}
    Let $f$ be an irreducible integer polynomial with positive leading coefficient. For $n\sim U[N]$, the random variable 
    \begin{equation*}
        \frac{\omega(f(n)) - \log\log N}{\sqrt{\log\log N}}
    \end{equation*}
    converges in distribution to $\mathcal{N}(0, 1)$ as $N\rightarrow\infty$. 
\end{thm}

\begin{thm}[Banks and Shparlinski~\cite{B-S}, Theorem 3] \label{Beatty}
    Let $\alpha > 0, \beta\in\mathbb{R}$. For $n\sim U[N]$, the random variable 
    \begin{equation*}
        \frac{\omega(\floor{\alpha n + \beta}) - \log\log N}{\sqrt{\log\log N}}
    \end{equation*}
    converges in distribution to $\mathcal{N}(0, 1)$ as $N\rightarrow\infty$. 
\end{thm}

Banks and Shparlinski in fact assumed in Theorem~\ref{Beatty} that $\alpha$ is irrational, but the corresponding statement for rational choices of $\alpha$ follows easily from Theorem~\ref{integer polynomial}. 

Crn{\v{c}}evi{\'c}, Hern{\'a}ndez, Rizk, Sereesuchart and Tao~\cite{pairwise-independence} extended Theorem~\ref{Beatty} by studying the joint distribution of $\omega(n)$ and $\omega(\floor{\alpha n})$, for irrational values of $\alpha$. In particular, they show that $\omega(n)$ and $\omega(\floor{\alpha n})$ are asymptotically independent. 

\begin{thm}[Crn{\v{c}}evi{\'c}, Hern{\'a}ndez, Rizk, Sereesuchart and Tao~\cite{pairwise-independence}, Theorem A] \label{pairwise}
    Let $\alpha > 0$ be irrational. For $n\sim U[N]$, the random variable 
    \begin{equation*}
        \left(\frac{\omega(n) - \log\log N}{\sqrt{\log\log N}},  \frac{\omega(\floor{\alpha n}) - \log\log N}{\sqrt{\log\log N}}\right)
    \end{equation*}
    converges in distribution to the standard two-dimensional Gaussian $\mathcal{N}(0, I_2)$ as $N\rightarrow\infty$. 
\end{thm}

The assumption that $\alpha$ is irrational is necessary here. Indeed, if $\alpha = a/b$ is rational, then whenever $b\mid n$, $\floor{\alpha n} = \alpha n$ shares all but at most $\omega(ab)$ prime divisors with $n$. Therefore, with positive probability $b^{-1} + o(1)$, 
\begin{equation*}
    \omega(\floor{\alpha n}) = \omega(n) + O(1). 
\end{equation*}
Hence the random variable 
\begin{equation*}
    \left(\frac{\omega(n) - \log\log N}{\sqrt{\log\log N}},  \frac{\omega(\floor{\alpha n}) - \log\log N}{\sqrt{\log\log N}}\right)
\end{equation*}
cannot converge to a continuous two-dimensional distribution. 

Theorem~\ref{pairwise} may be viewed as an assertion of the multiplicative independence of $n$ and $\floor{\alpha n}$ for irrational $\alpha$. Earlier results of this form include the classical work of Watson~\cite{coprime}, who showed that $n$ and $\floor{\alpha n}$ are coprime with probability $6/\pi^2 + o(1)$, the probability that two independently chosen integers are coprime. 

Our main theorem will be a multivariate generalisation of Theorem~\ref{Beatty} concerning the joint distribution of $\omega(\floor{\alpha_1 n + \beta_1}), \dots, \omega(\floor{\alpha_k n + \beta_k})$ whenever the pairwise ratios $\alpha_i/\alpha_j$ are irrational for all $i\neq j$. 

\begin{thm} \label{main}
    Let $\alpha_1, \dots, \alpha_k > 0$, where $\alpha_i/\alpha_j$ is irrational whenever $i\neq j$. Let $\beta_1, \dots, \beta_k$ be real numbers. For $n\sim U[N]$, the random variable 
    \begin{equation*}
        \left(\frac{\omega(\floor{\alpha_1 n + \beta_1}) - \log\log N}{\sqrt{\log\log N}}, \dots, \frac{\omega(\floor{\alpha_k n + \beta_k}) - \log\log N}{\sqrt{\log\log N}}\right)
    \end{equation*}
    converges in distribution to the standard $k$-dimensional Gaussian $\mathcal{N}(0, I_k)$ as $N\rightarrow\infty$. 
\end{thm}

Theorem~\ref{main} admits Theorem~\ref{pairwise} as a bivariate special case. As remarked above in the context of Theorem~\ref{pairwise}, the irrationality of $\alpha_i/\alpha_j$ for $i\neq j$ is necessary for the conclusion of Theorem~\ref{main} to hold for all choices of $\beta_1, \dots, \beta_k$. 

Multivariate generalisations of the Erd\H{o}s--Kac theorem have been studied in \cite{Tanaka}, which established a multivariate version of Theorem~\ref{integer polynomial}, and in \cite{multivariate} in a general setting. In particular, Theorem 2.1 of~\cite{multivariate} implies special cases of Theorem~\ref{main} where $\alpha_1, \dots, \alpha_k$ are rationally independent and satisfy suitable Diophantine type conditions. However, the assumptions of Theorem 2.1 of~\cite{multivariate} fail for general choices of parameters in Theorem~\ref{main}. 

We end by discussing the quantitative versions of the aforementioned Erd\H{o}s--Kac laws. Such quantitative bounds are often available under Diophantine type assumptions on the relevant irrational parameters. Here, we discuss the prospects for universal quantitative bounds on the rates of convergence without any such assumptions on the irrational parameters. 

The failure of Theorem~\ref{pairwise} for rational values of $\alpha$ may be leveraged to violate any quantitative bounds for Theorem~\ref{pairwise} by choosing $\alpha$ to be sufficiently Liouville-like. As such, the rate of convergence in Theorem~\ref{pairwise} and the more general Theorem~\ref{main} may be arbitrarily slow. 

In the case of Theorem~\ref{Beatty} for Beatty sequences $\floor{\alpha n + \beta}$, it was believed~\cite{B-S} that quantitative bounds require Diophantine type assumptions on $\alpha$. We show that in fact a universal quantitative bound may be given in this setting. 

\begin{thm} \label{Beatty quant}
    Let $\alpha > 0$, $\beta\in\mathbb{R}$. For $n\sim U[N]$, the Kolmogorov distance (see Section~\ref{quant conv}, Definition~\ref{Kolmogorov dist}) between the random variable 
    \begin{equation*}
        \frac{\omega(\floor{\alpha n + \beta}) - \log\log N}{\sqrt{\log\log N}}
    \end{equation*}
    and the standard Gaussian $\mathcal{N}(0, 1)$ is bounded above by $O_{\alpha, \beta}\left(\frac{\log\log\log N}{\sqrt{\log\log N}}\right)$ as $N\rightarrow\infty$, where the implied constant may depend on $\alpha$ and $\beta$. 
\end{thm}

Finally, we contrast this quantitative result for Beatty sequences with the case of $\floor{f(n)}$ for a higher-degree polynomial $f$ with an irrational non-constant coefficient (``a generalised polynomial"). An Erd\H{o}s--Kac law for general $\floor{f(n)}$ of this form in fact remains open, and may be viewed a common generalisation of Theorem~\ref{Beatty} for Beatty sequences and Theorem~\ref{integer polynomial} for integer polynomials. 

\begin{conj} \label{gen poly conj}
    Let $f$ be a polynomial with at least one irrational non-constant coefficient and a positive leading coefficient. For $n\sim U[N]$, the random variable 
    \begin{equation*}
        \frac{\omega(\floor{f(n)}) - \log\log N}{\sqrt{\log\log N}}
    \end{equation*}
    converges in distribution to $\mathcal{N}(0, 1)$ as $N\rightarrow\infty$. 
\end{conj}

In the case where the leading coefficient of the polynomial $f$ is irrational and of finite type, Conjecture~\ref{gen poly conj} may be answered affirmatively by combining Weyl's inequality (see~\cite[Lemma 2.4]{Vaughan}) with Fourier theoretic methods employed in our proof of Theorem~\ref{main}. In particular, this covers (Lebesgue) almost all choices of the leading coefficient and all algebraic irrational choices of the leading coefficient. 

Assuming that Conjecture~\ref{gen poly conj} holds, we show that the rate of convergence in Conjecture~\ref{gen poly conj} may be arbitrarily slow for non-linear $f$, in contrast with Theorem~\ref{Beatty quant} for the linear case of Beatty sequences. 

\begin{thm} \label{gen poly quant}
    For any sequence $\eta_N\searrow 0$ as $N\rightarrow\infty$, there exists a polynomial $f$ with at least one irrational non-constant coefficient and a positive leading coefficient, such that the following holds. For $n\sim U[N]$, the Kolmogorov distance between the random variable 
    \begin{equation*}
        \frac{\omega(\floor{f(n)}) - \log\log N}{\sqrt{\log\log N}}
    \end{equation*}
    and the standard Gaussian $\mathcal{N}(0, 1)$ is not $O(\eta_N)$ as $N\rightarrow\infty$. 
\end{thm}

All results here remain valid if, instead of $\omega$, we consider the total number $\Omega$ of prime divisors counted with multiplicity. Our arguments remain valid over short intervals of polynomial length. 

\textbf{Organisation of the paper.} Section~\ref{proof sketch} outlines the proof of Theorem~\ref{main}, which is carried out in Sections~\ref{lin alg}, \ref{moment decomp sec}, \ref{type B and D}, \ref{type A and C} and~\ref{assembly}. Section~\ref{quant conv} discusses Theorem~\ref{Beatty quant} and Conjecture~\ref{gen poly conj} on quantitative bounds on the rate of convergence. 

\textbf{Notation.} We use $f\ll g$ to denote $f = O(g)$, $f\lll g$ to denote $f = o(g)$, and $f\asymp g$ to denote $f = \Theta(g)$. All asymptotic relations are defined with respect to the limit as the parameter $N$ tends to infinity. All implied constants will be absolute unless otherwise stated. Despite the central role of floor functions in this paper to give positive integers whose prime factors are studied, we omit floor functions of real cutoff parameters where appropriate for the sake of clarity. Throughout, $p$ (with or without subscripts) denotes a prime. 

Given a vector $\mu\in\mathbb{R}^k$ and a positive-definite matrix $\Sigma\in \mathbb{R}^{k\times k}$, we denote the multivariate normal distribution with mean $\mu$ and covariance matrix $\Sigma$ by $\mathcal{N}(\mu, \Sigma)$. In particular, $\mathcal{N}(0, I_k)$ denotes the standard $k$-dimensional normal distribution, where $I_k$ denotes the $k\times k$ identity matrix. 

Given a probability distribution $\nu$, $\mathbb{P}_{n\sim \nu}(\cdot)$ and $\mathbb{E}_{n\sim \nu}(\cdot)$ denote the probability and expectation with respect to a random variable $n$ drawn from $\nu$, respectively. 

\section{Outline of the proof} \label{proof sketch}

We first give an outline informally sketching the ideas behind our proof of Theorem~\ref{main}, which will be presented in Sections~\ref{lin alg}, \ref{moment decomp sec}, \ref{type B and D}, \ref{type A and C} and~\ref{assembly}. The details presented here may differ from the more technical treatment to follow. 

To show the convergence to the multi-dimensional Gaussian asserted by Theorem~\ref{main}, we shall employ the method of moments. To do so, we rely on the classical result of probability theory that the distributional convergence to $\mathcal{N}(0, I_k)$ follows from the appropriate convergences of all mixed moments (see, for example, \cite{moment-det}, Section 30). 

\begin{thm} \label{moment-det}
    Let $X_N = (X_{1, N}, \dots, X_{k, N})$ be a sequence of random variables taking values in $\mathbb{R}^k$, indexed by a positive integer $N$. Let $Z = (Z_1, \dots, Z_k)\sim \mathcal{N}(0, I_k)$. If, for any non-negative integers $\ell_1, \dots, \ell_k$, the mixed moment
    \begin{equation*}
        \EE{X_{1, N}^{\ell_1}\cdots X_{k, N}^{\ell_k}}
    \end{equation*}
    is well-defined and converges to $\EE{Z_1^{\ell_1}\cdots Z_k^{\ell_k}}$ as $N\rightarrow\infty$, then $X_N$ converges in distribution to $\mathcal{N}(0, I_k)$ as $N\rightarrow\infty$. 
\end{thm}

Appealing to Theorem~\ref{moment-det} reduces Theorem~\ref{main} to the computation of the mixed moments
\begin{equation} \label{sketch main moment}
    \mathbb{E}_{n\sim U[N]}\left(\prod_{i = 1}^k(\omega(\floor{\alpha_i n + \beta_i}) - \log\log N)^{\ell_i}\right). 
\end{equation}
Following the approach of Granville and Soundararajan~\cite{method-of-moments-simplified}, we decompose $\omega(\floor{\alpha_i n + \beta_i}) - \log\log N$ as
\begin{equation*}
    \omega(\floor{\alpha_i n + \beta_i}) - \log\log N = \sum_{p\leq R}\left(1_{p\mid \floor{\alpha_i n + \beta_i}} - \frac{1}{p}\right) + O\left(\frac{\log N}{\log R}\right), 
\end{equation*}
where $R$ is a cutoff chosen so that $1\lll \frac{\log N}{\log R}\lll \sqrt{\log\log N}$. This decomposition in turn decomposes the moment (\ref{sketch main moment}) into terms of the form
\begin{equation} \label{sketch red moment}
    \mathbb{E}_{n\sim U[N]}\left(\prod_{i = 1}^k\prod_{j = 1}^{\ell_i}(1_{p_{ij}|\floor{\alpha_i n + \beta_i}} - p_{ij}^{-1})\right), 
\end{equation}
for not necessarily distinct primes $p_{ij}\leq R$, where $i\in [k]$ and $j\in [\ell_i]$. In the spirit of Isserlis' theorem, the main contribution to (\ref{sketch main moment}) comes from the case where for each $i$, the primes $p_{ij}$ come in pairs. This decomposition of the mixed moment (\ref{sketch main moment}) into terms of the form (\ref{sketch red moment}) will be discussed in more detail in Section~\ref{moment decomp sec}. 

We shall employ Fourier analytic techniques to estimate (\ref{sketch red moment}). To do so, we write (\ref{sketch red moment}) informally as
\begin{equation} \label{sketch integral 1}
    \int \left(\frac{1}{N}\sum_{n = 1}^N\delta_n(x)\right)\cdot\left(\prod_{i = 1}^k\prod_{j = 1}^{\ell_i}(1_{p_{ij}|\floor{\alpha_i x + \beta_i}} - p_{ij}^{-1})\right)dx, 
\end{equation}
which we compute using a version of Plancherel's identity. The analytically precise treatment will be given in Subsections~\ref{BD FT} and~\ref{Fourier}. 

As the function $x\mapsto 1_{p_{ij}|\floor{\alpha_i x + \beta_i}} - p_{ij}^{-1}$ is $p_{ij}/\alpha_i$-periodic, its Fourier transform is a sum of discrete masses supported on $(2\pi\alpha_i/p_{ij})\mathbb{Z}$. On the other hand, the Fourier transform of $\frac{1}{N}\sum_{n = 1}^N\delta_n(x)$ is concentrated within $\sim N^{-1}$ of $2\pi\mathbb{Z}$. This application of the Fourier transform therefore allows us to directly leverage the rational dependencies between $\alpha_i$'s. Structurally, applying a Fourier transform allows us to write (\ref{sketch integral 1}) as a sum
\begin{equation*}
    \sum_{(m_{ij})\in \prod (p_{ij}^{-1}\mathbb{Z})}\theta\left(\sum_{i = 1}^k\sum_{j = 1}^{\ell_i}\alpha_i m_{ij}\right) \prod_{i = 1}^k\prod_{j = 1}^{\ell_i}\frac{\phi(m_{ij})}{p_{ij}}, 
\end{equation*}
over tuples $(m_{ij})$ of $m_{ij}\in p_{ij}^{-1}\mathbb{Z}$, where $\theta$ is concentrated within $\sim N^{-1}$ of $\mathbb{Z}$, and $\phi$ decays to zero at infinity. As such, the main contribution originates from the tuples $(m_{ij})$ of $m_{ij}\in p_{ij}^{-1}\mathbb{Z}$ for which $\sum_{i = 1}^k\sum_{j = 1}^{\ell_i}\alpha_i m_{ij}$ is close to an integer. A suitable characterisation of such tuples will be given in Section~\ref{lin alg}. Following this characterisation, Section~\ref{lin alg} also gives the precise treatment which reduces Theorem~\ref{main} to the moment estimate Proposition~\ref{reduced main}. 

Heuristically, for generic values of $\alpha_1, \dots, \alpha_k$, 
\begin{equation*}
    \sum_{i = 1}^k\sum_{j = 1}^{\ell_i}\alpha_i m_{ij} = \sum_{i = 1}^k\left(\sum_{j = 1}^{\ell_i}m_{ij}\right)\alpha_i 
\end{equation*}
may only be within $\sim N^{-1}$ of an integer when $\sum_j m_{ij} = 0$ for each $i\in [k]$. This ``entry-wise" vanishing condition may be interpreted as the source of independence of $\omega(\floor{\alpha_i n + \beta_i})$ across different choices of $i\in [k]$. 

However, to apply Plancherel's identity to (\ref{sketch integral 1}), we need to ensure that the Fourier transform of 
\begin{equation*}
    \prod_{i = 1}^k\prod_{j = 1}^{\ell_i}(1_{p_{ij}|\floor{\alpha_i x + \beta_i}} - p_{ij}^{-1})
\end{equation*}
is absolutely convergent. Unfortunately, this fails to be the case, due to the discontinuous nature of the indicator function $1_{p_{ij}|\floor{\alpha_i x + \beta_i}}$. We shall remedy this situation by replacing this indicator by a suitable continuous approximant. As it turns out, two different approximants are needed. These two approximation of the indicator function and the subsequent Fourier analysis will be carried out in Sections~\ref{type B and D} and~\ref{type A and C}. 

Finally, the estimation of the mixed moment (\ref{sketch main moment}) will be done in Section~\ref{assembly}, completing the proof of Theorem~\ref{main}. 

\section{Linear algebra over the rationals} \label{lin alg}

Throughout the proof of Theorem~\ref{main}, we treat $\alpha_1, \dots, \alpha_k$ as absolute constants, and implicitly assume that $N$ is sufficiently large in an absolute sense. 

For a real number $x$, let $\|x\|_{\mathbb{Z}} = \min_{m\in \mathbb{Z}} |x - m|$ denote its distance to the closest integer. For primes $p_{ij}\leq R$ (where $i\in [k]$, $j\in [\ell_i]$), we wish to characterise the tuples $(m_{ij})$ of $m_{ij}\in p_{ij}^{-1}\mathbb{Z}$ satisfying $\|\sum_{i = 1}^k\sum_{j = 1}^{\ell_i}\alpha_im_{ij}\|_{\mathbb{Z}}\leq N^{-1/4}$. We shall use $\sum_{ij}$ to denote this sum over $\sum_{i = 1}^k\sum_{j = 1}^{\ell_i}$. 

\begin{df}
    For coprime integers $a\in\mathbb{Z}, b\in \mathbb{Z}^+$, let the \emph{height} of the rational number $a/b$ be $\max(|a|, b)$. Let $\mathcal{Q}(h)$ denote the set of rational numbers of height at most $h$. 
\end{df}

Note that if $x\in \mathcal{Q}(h_1), y\in \mathcal{Q}(h_2)$, then $x \pm y, xy, x/y$ are all in $\mathcal{Q}(2h_1h_2)$. 

\begin{lem} \label{key lin alg}
    Let $R, J, L$ be functions of $N$ tending to infinity with $N$, such that $R \ll J\leq N^{o(L^{-1})}$. Let $p_{ij}\leq R$ where $i\in [k]$, and $j \in [\ell_i]$. Let $\ell = \ell_1 + \dots + \ell_k$. 
    
    There exists $\gamma_1, \dots, \gamma_k\in \mathcal{Q}\left(J^{O(L)}\right)$ dependent only on $N$ such that
    \begin{enumerate}
        \item whenever $L \geq \ell$, if $m_{ij}\in p_{ij}^{-1}\mathbb{Z}\cap(-J, J)$ satisfies $\|\sum_{ij}\alpha_im_{ij} \|_{\mathbb{Z}}\leq N^{-1/4}$, then $\sum_{ij}\gamma_im_{ij}\in\mathbb{Z}$, 
        \item $\gamma_i\rightarrow \alpha_i$ as $N\rightarrow\infty$. 
    \end{enumerate}
\end{lem}

This lemma essentially follows from the fact that all linear algebraic operations concerning a matrix of absolutely bounded dimension involve only an absolutely bounded number of field operations. 

\begin{proof}
    In this proof, the implied constants are absolute and may change from line to line. Let
    $$B = \left\{(m_1, \dots, m_k, m)\middle| m_i\in \mathcal{Q}\left(J^{2L}\right), m\in \mathbb{Z}\text{ and }\left|\sum_i\alpha_im_i + m\right|\leq N^{-1/4}\right\}\subseteq \mathbb{Q}^{k + 1}.$$
    Let $m^{(1)}, \dots, m^{(r)}\in B$ be a basis of $\operatorname{span}(B)$. Note that $r\leq k + 1$. Let $M$ be the $r\times (k + 1)$ matrix with $m^{(1)}, \dots, m^{(r)}$ as row vectors. Let $u = (\alpha_1, \dots, \alpha_k, 1)$. By definition, $\|M\cdot u\|_\infty\leq N^{-1/4}$. 
    
    Viewing $M$ as a linear map, we may find a basis $v_1, \dots, v_{k + 1}$ of $\mathbb{Q}^{k + 1}$ such that $M\cdot v_1, \dots, M\cdot v_r$ form a basis of $\operatorname{im} M$ and $v_{r + 1}, \dots, v_{k + 1}$ form a basis of $\ker M$. This basis may be computed using an absolutely bounded number of field operations. Therefore, we may assume that the entries of $v_i$ are in $\mathcal{Q}\left(J^{C_1L}\right)$, for some absolute constant $C_1$. Decompose $u$ as $u = \sum_{i = 1}^{k + 1}c_i v_i$ for real coefficients $c_i$. 
    
    As the entries of $M$ and $v_1, \dots, v_r$ are in $\mathcal{Q}\left(J^{O(L)}\right)$, the entries of $M\cdot v_1, \dots, M\cdot v_r$ are also in $\mathcal{Q}\left(J^{O(L)}\right)$. Using an absolutely bounded number of field operations, we may find a dual basis $v_1', \dots, v_r'$ to the basis $M\cdot v_1, \dots, M\cdot v_r$ of $\operatorname{im} M$. Therefore, the entries of $v_1', \dots, v_r'$ are in $\mathcal{Q}\left(J^{O(L)}\right)$. Since $M\cdot u = \sum_{i = 1}^r c_i M\cdot v_i\in \operatorname{im} M$, we have $c_i = v_i'^T\cdot M\cdot u$ for $i \in [r]$. Therefore, $|c_i|\leq O\left(\|v_i'\|_\infty \|M\cdot u\|_\infty\right)\leq O\left(J^{O(L)}N^{-1/4}\right)$ for $i\in [r]$. 
    
    By similarly taking a dual basis for $v_1, \dots, v_{k + 1}$, we may show, for $i = r + 1, \dots, k + 1$, that $|c_i|\leq J^{C_2L}$ for some absolute constant $C_2$. Taking $C_3 = \max(C_1, C_2) + 1$, we may take $c_{r + 1}', \dots, c_{k + 1}'\in \mathcal{Q}\left(J^{C_3L}\right)$ such that $|c_i' - c_i|\leq J^{-C_3L}$. Let $u' = \sum_{i = r + 1}^{k + 1} c_i' v_i$. Then $u'$ has entries in $\mathcal{Q}\left(J^{O(L)}\right)$ and 
    $$\|u' - u\|_\infty \ll \left(\max_{i\leq r}|c_i| + \max_{i > r}|c_i' - c_i|\right)\max \|v_i\|_\infty\ll J^{C_1L - C_3L} = o(1).$$
    Take $\gamma_1, \dots, \gamma_k$ so that $u'':=(\gamma_1, \dots, \gamma_k, 1)$ is parallel to $u'$. We have $\gamma_i\in \mathcal{Q}\left(J^{O(L)}\right)$ and $\|(\gamma_1, \dots, \gamma_k, 1) - u\|_\infty = o(1)$. That is, $\gamma_i\rightarrow \alpha_i$ as $N\rightarrow\infty$. Since $u''\parallel u'\in\ker M$, $u''\cdot m^{(1)} =\dots = u''\cdot m^{(r)} = 0$. As $m^{(1)}, \cdots, m^{(r)}$ form a basis of $\operatorname{span}(B)$, $u''\cdot v = 0$ for any $v\in B$. 
    
    If $\ell\leq L$, take any $m_{ij}\in p_{ij}^{-1}\mathbb{Z}\cap(-J, J)$ such that $\|\sum_{ij}\alpha_im_{ij} \|_{\mathbb{Z}}\leq N^{-1/4}$. Let $m_i = \sum_j m_{ij}\in \mathcal{Q}\left(2^{\ell_i - 1}J^{\ell_i}\right)\subseteq \mathcal{Q}\left(J^{2L}\right)$. Therefore, $(m_1, \dots, m_k, m)\in B$ for some $m\in \mathbb{Z}$. As a result, $u''\cdot (m_1, \dots, m_k, m) = 0$ and $\sum_{ij}\gamma_im_{ij} = \sum_i \gamma_i m_i = -m\in\mathbb{Z}$. 
\end{proof}

\begin{df} \label{params}
    Fix $R, J, L$ as functions of $N$ such that
    \begin{equation*}
        1\lll L\lll \frac{\log N}{\log J}, \frac{\log J}{\log R} \lll \log\log\log N. 
    \end{equation*}
\end{df}

Note that $\gamma_i$ depends only on $N$, and in particular not on the primes $p_{ij}$ we consider. We assume that $\gamma_1, \dots, \gamma_k > 0$, as would be the case for sufficiently large $N$. 

\begin{df}
    Let $\mathcal{B}'_N$ be the union of the sets of prime factors of the numerators and denominators of $\gamma_1, \dots, \gamma_k$. Let $\mathcal{B}_N$ denote the union of $\mathcal{B}_N'$ with the set of primes smaller or equal to $\log N$. 
\end{df}

The set $\mathcal{B}_N$ consists of a small number of potentially problematic prime divisors. Removing these primes allows us to obtain quantitative equidistribution bounds needed for precise mixed moment estimates. 

\begin{df}
    Let $\mathcal{P}_N$ denote the set of primes $p\leq R$ not in $\mathcal{B}_N$. Let 
    \begin{equation*}
        \omega_N'(n) := \sum_{p\in \mathcal{P}_N}1_{p\mid n}
    \end{equation*}
    approximate $\omega$. Let
    \begin{equation*}
        \widetilde{\omega_N'}(n) := \omega_N'(n) - \sum_{p\in \mathcal{P}_N} p^{-1} = \sum_{p\in \mathcal{P}_N}(1_{p\mid n} - p^{-1})
    \end{equation*}
    approximate $\omega - \log\log N$. 
\end{df}

Note that $|\mathcal{B}_N|\ll L\log J + \log N\ll \log N$. In particular, by Mertens' estimate
\begin{equation} \label{bad primes estimate}
    \sum_{p\in \mathcal{B}_N}\frac{1}{p}\leq \sum_{\text{the first }|\mathcal{B}_N|\text{ primes } p}\frac{1}{p}\ll \log\log |\mathcal{B}_N|\ll \log\log\log N. 
\end{equation}
Therefore, again by Mertens' estimate, 
\begin{equation} \label{good primes estimate}
    \sum_{p\in \mathcal{P}_N} \frac{1}{p} = \log\log R + O(\log\log\log N) = \log\log N + o(\sqrt{\log\log N}). 
\end{equation}

We now show that on average, for $n\sim U[N]$, $\omega'_N(\floor{\alpha_i n + \beta_i})\leq \omega(\floor{\alpha_i n + \beta_i})$ forms a good approximant of $\omega(\floor{\alpha_i n + \beta_i})$. 

\begin{lem} \label{omega and omega'}
    The random variables
    \begin{equation*}
        \frac{\omega(\floor{\alpha_i n + \beta_i}) - \omega'_N(\floor{\alpha_i n + \beta_i})}{\sqrt{\log\log N}}\rightarrow 0
    \end{equation*}
    in probability. 
\end{lem}

\begin{proof}
    It suffices to show convergence in expectation. By (\ref{bad primes estimate}), we have
    \begin{align*}
        &\mathbb{E}_{n\sim U[N]}\left(\omega(\floor{\alpha_i n + \beta_i}) - \omega'_N(\floor{\alpha_i n + \beta_i})\right) \\
        &\leq \mathbb{E}_{n\sim U[N]}\left(\sum_{p > R} 1_{p\mid \floor{\alpha_i n + \beta_i}}\right) + \mathbb{E}_{n\sim U[N]}\left(\sum_{p \in \mathcal{B}_N} 1_{p\mid \floor{\alpha_i n + \beta_i}}\right)\\
        &\leq \frac{\log N}{\log R} + \sum_{p\in \mathcal{B}_N}\mathbb{P}_{n\sim U[N]}\left(p\mid\floor{\alpha_i n + \beta_i}\right)\\
        &\ll \frac{\log N}{\log R} + \sum_{p\in \mathcal{B}_N}\left(\frac{1}{p} + \frac{1}{N}\right)\\
        &\ll \frac{\log N}{\log R} + \log\log\log N\\
        &\leq o\left(\sqrt{\log\log N}\right), 
    \end{align*}
    where we use the estimate
    \begin{equation*}
        \mathbb{P}_{n\sim U[N]}\left(p\mid\floor{\alpha_i n + \beta_i}\right)\ll \frac{1}{p} + \frac{1}{N}. \qedhere
    \end{equation*}
\end{proof}

\begin{cor} \label{reduction}
    The random variables
    \begin{equation*}
        \frac{\omega(\floor{\alpha_i n + \beta_i}) - \log\log N}{\sqrt{\log\log N}} - \frac{\widetilde{\omega'_N}(\floor{\alpha_i n + \beta_i})}{\sqrt{\log\log N}}\rightarrow 0
    \end{equation*}
    in probability. 
\end{cor}

\begin{proof}
    This follows by combining Lemma~\ref{omega and omega'} with (\ref{good primes estimate}). 
\end{proof}

Corollary~\ref{reduction} reduces Theorem~\ref{main} to showing that 
\begin{equation*}
    \left(\frac{\widetilde{\omega'_N}(\floor{\alpha_1 n + \beta_1})}{\sqrt{\log\log N}}, \dots, \frac{\widetilde{\omega'_N}(\floor{\alpha_k n + \beta_k})}{\sqrt{\log\log N}}\right)\rightarrow_d \mathcal{N}(0, I_k)
\end{equation*}
in distribution. In light of Theorem~\ref{moment-det}, it further suffices to show the convergence of mixed moments. In other words, we have reduced Theorem~\ref{main} to showing the following moment estimate. 

\begin{prop} \label{reduced main}
    For any set of non-negative integers $\ell_1, \dots, \ell_k$, we have
    \begin{equation*}
        \mathbb{E}_{n\sim U[N]}\left(\prod_{i = 1}^k \left(\frac{\widetilde{\omega'_N}(\floor{\alpha_i n + \beta_i})}{\sqrt{\log\log N}}\right)^{\ell_i}\right)\rightarrow \mathbb{E}_{Z\sim\mathcal{N}(0, I_k)}(Z_1^{\ell_1}\cdots Z_k^{\ell_k})
    \end{equation*}
    as $N\rightarrow \infty$. 
\end{prop}

Sections~\ref{moment decomp sec}, \ref{type B and D}, \ref{type A and C} and~\ref{assembly} are dedicated to the proof of Proposition~\ref{reduced main}, which completes the proof of Theorem~\ref{main}. We shall now treat $\ell := \ell_1 + \dots + \ell_k$ as an absolute constant, and assume that $N$ is sufficiently large with respect to $\ell$, so that $L\geq \ell$.

\section{Decomposing the moment} \label{moment decomp sec}

Using
\begin{equation*}
    \widetilde{\omega'_N}(n) = \sum_{p\in \mathcal{P}_N} \left(1_{p\mid n} - \frac{1}{p}\right), 
\end{equation*}
we decompose the mixed moment
\begin{equation*}
    \mathbb{E}_{n\sim U[N]}\left(\prod_{i = 1}^k \left(\frac{\widetilde{\omega'_N}(\floor{\alpha_i n + \beta_i})}{\sqrt{\log\log N}}\right)^{\ell_i}\right)
\end{equation*}
as
\begin{equation*}
    \sum_{p_{ij}\in \mathcal{P}_N}\mathbb{E}_{n\sim U[N]}\left(\prod_{i = 1}^k\prod_{j = 1}^{\ell_i} \frac{1_{p_{ij}\mid \floor{\alpha_i n + \beta_i}} - p_{ij}^{-1}}{\sqrt{\log\log N}}\right), 
\end{equation*}
where we sum over tuples $(p_{ij})$ of primes $p_{ij}$ ($i\in [k]$ and $j\in [\ell_i]$). We use $\prod_{ij}$ to denote this product over $\prod_{i = 1}^k\prod_{j = 1}^{\ell_i}$. We now study the summand for each tuple $(p_{ij})$ of primes $p_{ij}\in \mathcal{P}_N$. 

\begin{df}
    Let 
    \begin{equation*}
        E = E(p_{ij}):= \mathbb{E}_{n\sim U[N]}\left(\prod_{ij} (1_{p_{ij}\mid \floor{\alpha_i n + \beta_i}} - p_{ij}^{-1})\right). 
    \end{equation*}
\end{df}

In this language, we have
\begin{equation} \label{moment decomposition}
    \mathbb{E}_{n\sim U[N]}\left(\prod_{i = 1}^k \left(\frac{\widetilde{\omega'_N}(\floor{\alpha_i n + \beta_i})}{\sqrt{\log\log N}}\right)^{\ell_i}\right) = \frac{1}{(\log\log N)^{\ell/2}}\sum_{p_{ij}\in \mathcal{P}_N}E(p_{ij}). 
\end{equation}
To study the contribution of each tuple $(p_{ij})$ to this sum, we categorise them into four types. 
\begin{df}
    We call a tuple $(p_{ij})$ of primes $p_{ij}\in \mathcal{P}_N$ ($i\in [k]$ and $j\in [\ell_i]$), 
    \begin{enumerate}
        \item type A if the $p_{ij}$s come distinct in pairs, and furthermore if $p_{ij} = p_{i'j'}$, then $i = i'$, 
        \item type B if the $p_{ij}$s come distinct in pairs, but there exists $i\neq i'$ such that $p_{ij} = p_{i'j'}$, 
        \item type C if some prime occurs only once, 
        \item type D, otherwise.
    \end{enumerate}
\end{df}

The main contribution to 
\begin{equation*}
    \frac{1}{(\log\log N)^{\ell/2}}\sum_{p_{ij}\in \mathcal{P}_N}E(p_{ij}). 
\end{equation*}
comes from type A tuples. We seek to estimate
\begin{equation*}
    \sum_{\text{type A }(p_{ij})}E(p_{ij}), 
\end{equation*}
and to upper bound
\begin{equation*}
    \sum_{\text{type B, C, D }(p_{ij})}|E(p_{ij})|. 
\end{equation*}

Roughly speaking, type B terms are suppressed by the irrationality of $\alpha_i/\alpha_{i'}$ relative to type A terms. Each type D term may be on the same order of magnitude as type A terms, but there are much fewer type D terms. Type C terms are negligible. 

We will directly upper bound $|E(p_{ij})|$ for type B, C, and D tuples, and relate $E(p_{ij})$ for type A tuples to a suitable type C terms. 

For type B, C, and D tuples, we shall estimate $|E|$ by taking a suitable Fourier transform. However, this approach is impeded by the discontinuity of 
\begin{equation*}
    1_{p_{ij}\mid \floor{\alpha_i n + \beta_i}} - p_{ij}^{-1}. 
\end{equation*}
We shall therefore replace this function with a suitable continuous variant before taking the Fourier transform. As it turns out, two different choices of this continuous replacement are needed. In both cases, it suffices to give a continuous approximant of $1_{[0, 1)}$, as we may decompose
\begin{equation*}
    1_{p\mid \floor{x}} = \sum_{m\in \mathbb{Z}}1_{[0, 1)}(x - pm), 
\end{equation*}
in terms of indicator functions of unit intervals. For type B and D tuples we make use of the bound
\begin{equation*}
    \left|1_{p\mid \floor{x}} - p^{-1}\right|\leq 1_{p\mid \floor{x}} + p^{-1} = \sum_{m\in \mathbb{Z}}1_{[0, 1)}(x - pm) + p^{-1}\leq \sum_{m\in \mathbb{Z}}1_{[-1, 1]}*1_{[-1, 1]}(x - pm) + p^{-1}. 
\end{equation*}
This corresponds to considering the crude majorant $1_{[-1, 1]}*1_{[-1, 1]}$ of $1_{[0, 1)}$. We shall treat the case of type B and D tuples in Section~\ref{type B and D}. 

For type C tuples, we instead approximate $1_{[0, 1)}$ by either $\epsilon^{-1}1_{[0, 1]}*1_{[0, \epsilon]}$ or $\epsilon^{-1}1_{[0, 1]}*1_{[-\epsilon, 0]}$ for a suitably chosen $\epsilon = \epsilon(N)> 0$. Note that neither forms a majorant or a minorant of $1_{[0, 1)}$. It will be important that the Fourier transform of the approximant we use vanishes on $\mathbb{Z}^{\neq 0}$. As such, we are unable to use tight majorants or minorants of $1_{[0, 1)}$ in this setting. The treatment in this case is more intricate and will be given in Section~\ref{type A and C}. 

Suitable type D estimates may in fact be derived alternatively from the type C estimates developed in Section~\ref{type A and C}, but type B estimates rely on the framework of Section~\ref{type B and D}. 

\section{Type B and D tuples} \label{type B and D}

\subsection{Smoothing the indicator} \label{indicator approximation 1}

\begin{df}
    Let $\chi := 1_{[-1, 1]}*1_{[-1, 1]}$ be a continuous upper bound of $1_{[0, 1)}$. Let 
    \begin{equation*}
        \chi_p(x) := \sum_{m\in \mathbb{Z}}\chi(x - pm), 
    \end{equation*}
    be the corresponding continuous upper bound of $1_{p\mid \floor{x}}$. 
\end{df}

Note that
\begin{equation*}
    \left|1_{p\mid \floor{x}} - p^{-1}\right|\leq \chi_p(x) + p^{-1}. 
\end{equation*}
Therefore, we have
\begin{equation} \label{BD E bound}
    |E|\leq \mathbb{E}_{n\sim U[N]}\left(\prod_{ij} (\chi_{p_{ij}}(\alpha_i n + \beta_i) + p_{ij}^{-1})\right). 
\end{equation}

\subsection{The Fourier transform} \label{BD FT}

As outlined in Section~\ref{proof sketch}, we shall evaluate the right-hand side of (\ref{BD E bound}) via a Fourier transform. 
\begin{df}
    Let $e(x)$ denote $e^{2\pi i x}$. Let
    \begin{equation*}
        \phi(y) := \widehat{\chi}(y) = \int \chi(x) e(-xy)dx \geq 0, 
    \end{equation*}
    and let 
    \begin{equation*}
        \theta(x) := N^{-1}\sum_{n = 1}^N e(nx). 
    \end{equation*}
\end{df}

Note that 
\begin{equation} \label{Fourier l infty estimate BD}
    \phi(m) \ll\min(1, m^{-1})^2, 
\end{equation}
and therefore, 
\begin{align} 
    \sum_{m\in p^{-1}\mathbb{Z}}\frac{1}{p}\phi(m)&\ll 1,\label{Fourier l1 estimate all BD} \\
    \sum_{m\in p^{-1}\mathbb{Z}, |m|\geq J}\frac{1}{p}\phi(m)&\ll J^{-1}. \label{Fourier l1 estimate tail BD}
\end{align}

Note that the restriction of $\frac{1}{p}\phi$ on $p^{-1}\mathbb{Z}$ forms the Fourier coefficients of the $p$-periodisation $\chi_p$ of $\chi$. Since 
\begin{equation*}
    \sum_{m\in p^{-1}\mathbb{Z}}\frac{1}{p}|\phi(m)| < \infty
\end{equation*}
is absolutely convergent, we have the Fourier inversion formula
\begin{equation} \label{Fourier inversion BD}
    \sum_{m\in p^{-1}\mathbb{Z}}\frac{\phi(m)}{p}e(mx) = \chi_p(x). 
\end{equation}

\begin{lem}
    We have
    \begin{align*}
        &\mathbb{E}_{n\sim U[N]}\left(\prod_{ij} (\chi_{p_{ij}}(\alpha_i n + \beta_i) + p_{ij}^{-1})\right)\\
        &= \sum_{(m_{ij})\in \prod (p_{ij}^{-1}\mathbb{Z})}\theta\left(\sum_{ij} \alpha_im_{ij}\right)\prod_{ij}\frac{(\phi(m_{ij}) + \delta_0(m_{ij}))e\left(\beta_im_{ij}\right)}{p_{ij}}, 
    \end{align*}
    where $\sum_{(m_{ij})\in \prod (p_{ij}^{-1}\mathbb{Z})}$ denotes the sum over all choices of the tuple $(m_{ij})_{ij}$ with $m_{ij}\in p_{ij}^{-1}\mathbb{Z}$. 
\end{lem}

\begin{proof}
    Starting on the right-hand side, we have, by Fubini's theorem and (\ref{Fourier inversion BD}), 
    \begin{align*}
        &\sum_{(m_{ij})\in \prod (p_{ij}^{-1}\mathbb{Z})}\theta\left(\sum_{ij} \alpha_im_{ij}\right)\prod_{ij}\frac{(\phi(m_{ij}) + \delta_0(m_{ij}))e\left(\beta_im_{ij}\right)}{p_{ij}}\\
        &= N^{-1}\sum_{(m_{ij})\in \prod (p_{ij}^{-1}\mathbb{Z})}\sum_{n = 1}^N e\left(n\sum_{ij} \alpha_im_{ij}\right)\prod_{ij}\frac{(\phi(m_{ij}) + \delta_0(m_{ij}))e\left(\beta_im_{ij}\right)}{p_{ij}}\\
        &= N^{-1}\sum_{n = 1}^N \prod_{ij}\sum_{m_{ij}\in p_{ij}^{-1}\mathbb{Z}}\frac{\phi(m_{ij}) + \delta_0(m_{ij})}{p_{ij}}e\left((\alpha_in + \beta_i)m_{ij}\right)\\
        &= N^{-1}\sum_{n = 1}^N \prod_{ij}(\chi_{p_{ij}}(\alpha_i n + \beta_i) + p_{ij}^{-1})\\
        &= \mathbb{E}_{n\sim U[N]}\left(\prod_{ij} (\chi_{p_{ij}}(\alpha_i n + \beta_i) + p_{ij}^{-1})\right). \qedhere
    \end{align*}
\end{proof}

Noting that $\phi(0) = 4$, we have
\begin{equation*}
    \phi + \delta_0\ll \phi. 
\end{equation*}
Note also that
\begin{equation*}
    |\theta(x)|\ll \frac{1}{\max(N\|x\|_\mathbb{Z}, 1)}. 
\end{equation*}
Combining these observations with (\ref{BD E bound}), we arrive at the following corollary of the preceding lemma. 

\begin{cor} \label{BD E cor}
    \begin{equation*}
        |E|\ll \sum_{(m_{ij})\in \prod (p_{ij}^{-1}\mathbb{Z})}\frac{1}{\max(N\|\sum_{ij} \alpha_im_{ij}\|_\mathbb{Z}, 1)}\prod_{ij}\frac{\phi(m_{ij})}{p_{ij}}. 
    \end{equation*}
\end{cor}

In light of Lemma~\ref{key lin alg}, we now show that it suffices to sum over $m_{ij}$ for which $\sum_{ij}\gamma_im_{ij}\in \mathbb{Z}$. 

\begin{df}
    Let 
    \begin{equation*}
        E' = E'(p_{ij}):= \sum_{\substack{(m_{ij})\in \prod (p_{ij}^{-1}\mathbb{Z})\\ \sum_{ij}m_{ij}\gamma_i\in\mathbb{Z}}}\prod_{ij}\frac{\phi(m_{ij})}{p_{ij}}. 
    \end{equation*}
\end{df}

\begin{lem}
    \begin{equation*}
        |E|\ll E' + J^{-1}. 
    \end{equation*}
\end{lem}

\begin{proof}
    We first deal with the case when some $|m_{ij}|\geq J$ in the sum in Corollary~\ref{BD E cor}. Here, by (\ref{Fourier l1 estimate all BD}) and (\ref{Fourier l1 estimate tail BD}), we have
    \begin{equation*}
        \sum_{\substack{(m_{ij})\in \prod (p_{ij}^{-1}\mathbb{Z})\\\text{some }|m_{ij}|\geq J}}\prod_{ij}\frac{\phi(m_{ij})}{p_{ij}}\ll \ell J^{-1}\ll J^{-1}, 
    \end{equation*}
    Therefore, it suffices to show that
    \begin{equation*}
        \sum_{\substack{(m_{ij})\in \prod (p_{ij}^{-1}\mathbb{Z}\cap (-J, J))\\ \sum_{ij}m_{ij}\gamma_i \notin\mathbb{Z}}}\frac{1}{\max(N\|\sum_{ij} \alpha_im_{ij}\|_\mathbb{Z}, 1)}\prod_{ij}\frac{\phi(m_{ij})}{p_{ij}}\ll J^{-1}. 
    \end{equation*}
    Assume that $N$ is sufficiently large so that $L\geq \ell$. By Lemma~\ref{key lin alg}, for any tuple $(m_{ij})$ in the preceding sum, we have $\|\sum_{ij}\alpha_im_{ij}\|_\mathbb{Z}\geq N^{-1/4}$. Therefore, by (\ref{Fourier l1 estimate all BD}), 
    \begin{align*}
        \sum_{\substack{(m_{ij})\in \prod (p_{ij}^{-1}\mathbb{Z}\cap (-J, J))\\ \sum_{ij}m_{ij}\gamma_i \notin\mathbb{Z}}}\frac{1}{\max(N\|\sum_{ij} \alpha_im_{ij}\|_\mathbb{Z}, 1)}\prod_{ij}\frac{\phi(m_{ij})}{p_{ij}}&\ll N^{-3/4}\sum_{\substack{(m_{ij})\in \prod (p_{ij}^{-1}\mathbb{Z}\cap (-J, J))\\ \sum_{ij}m_{ij}\gamma_i \notin\mathbb{Z}}}\prod_{ij}\frac{\phi(m_{ij})}{p_{ij}}\\
        &\ll N^{-3/4}, 
    \end{align*}
    which suffices since $J\leq N^{o(1)}$. 
\end{proof}

\subsection{Type D terms}

We use the following lemma to bound $E'$ for type D tuples. 

\begin{df}
    Let $\mathcal{R}_p\subseteq \mathbb{Q}$ denote the set of rationals $a/b$ where $p\nmid b$ (where $a, b$ are coprime integers). 
\end{df}

\begin{lem} \label{type D lem}
    For any $p\in\mathcal{P}_N$, and any $\gamma'_1, \dots, \gamma'_v\in \{\gamma_1, \dots, \gamma_k\}$, 
    \begin{equation*}
        \sum_{\substack{m_1, \dots, m_v\in p^{-1}\mathbb{Z}\\ \sum_u m_u\gamma'_u\in\mathcal{R}_p}}\prod_u\frac{\phi(m_u)}{p}\ll p^{-1}. 
    \end{equation*}
\end{lem}

\begin{proof}
    We have, by (\ref{Fourier l1 estimate all BD}), 
    \begin{align*}
        \sum_{\substack{m_1, \dots, m_v\in p^{-1}\mathbb{Z}\\ \sum_u m_u\gamma'_u\in\mathcal{R}_p}}\prod_u\frac{\phi(m_u)}{p}&= \sum_{m_1, \dots, m_{v - 1}\in p^{-1}\mathbb{Z}}\prod_{u = 1}^{v - 1}\frac{\phi(m_u)}{p}\sum_{\substack{m_v\in p^{-1}\mathbb{Z}\\ \sum_u m_u\gamma'_u\in\mathcal{R}_p}}\frac{\phi(m_v)}{p}\\
        &\leq \left(\sum_{m'\in p^{-1}\mathbb{Z}}\frac{\phi(m')}{p}\right)^{v - 1}\sup_{r\in\mathbb{Q}}\sum_{\substack{m\in p^{-1}\mathbb{Z}\\ m\gamma'_v\in\mathcal{R}_p + r}}\frac{\phi(m)}{p}\\
        &\ll p^{-1}\sup_{r\in\mathbb{Q}}\sum_{\substack{m\in p^{-1}\mathbb{Z}\\ m\gamma'_v\in\mathcal{R}_p + r}}\phi(m)
    \end{align*}
    The numerator and denominator of $\gamma_v'$ are coprime to $p\in\mathcal{P}_N$ by the definition of $\mathcal{P}_N$. Therefore, the permissible values of $m$ in the sum
    \begin{equation*}
        \sum_{\substack{m\in p^{-1}\mathbb{Z}\\ m\gamma'_v\in\mathcal{R}_p + r}}\phi(m)
    \end{equation*}
    are all congruent modulo 1. Therefore, in light of (\ref{Fourier l infty estimate BD}), 
    \begin{equation*}
        \sum_{\substack{m\in p^{-1}\mathbb{Z}\\ m\gamma'_v\in\mathcal{R}_p + r}}\phi(m)\ll 1
    \end{equation*}
    for any choice of $r\in\mathbb{Q}$, and the desired bound follows. 
\end{proof}

\begin{cor} \label{individual type D}
    For any type D tuple $(p_{ij})$, 
    \begin{equation*}
        E'\ll \prod_{ij}'p_{ij}^{-1}, 
    \end{equation*}
    where $\prod_{ij}'$ denotes the product over distinct $p_{ij}$'s. 
\end{cor}

\begin{proof}
    Recall that $E'$ is defined as
    \begin{equation*}
        E' = \sum_{\substack{(m_{ij})\in \prod (p_{ij}^{-1}\mathbb{Z})\\ \sum_{ij}m_{ij}\gamma_i\in\mathbb{Z}}}\prod_{ij}\frac{\phi(m_{ij})}{p_{ij}}. 
    \end{equation*}
    The desired bound follows from the preceding lemma once we note that for each prime $p$ in the tuple $(p_{ij})$, we only sum over tuples $(m_{ij})$ satisfying
    \begin{equation*}
        \sum_{ij \text{ s.t. } p_{ij} = p}\gamma_im_{ij}\equiv_{\mod 1} \sum_{ij \text{ s.t. } p_{ij}\neq p}\gamma_im_{ij}\in\mathcal{R}_p. \qedhere
    \end{equation*}
\end{proof}

\begin{cor} \label{type D sum}
    We have
    \begin{equation*}
        \left|\sum_{\text{type D }(p_{ij})}E(p_{ij})\right|\ll \left(\log\log N\right)^{\frac{\ell - 1}{2}}. 
    \end{equation*}
\end{cor}

\begin{proof}
    We have
    \begin{align*}
        \left|\sum_{\text{type D }(p_{ij})}E(p_{ij})\right|&\ll \sum_{\text{type D }(p_{ij})}E'(p_{ij}) + R^\ell J^{-1}\\
        &\ll \left(\sum_{\text{type D }(p_{ij})}\prod_{ij}'p_{ij}^{-1}\right) + 1\\
        &\ll \left(1 + \sum_{p\in\mathcal{P}_N} p^{-1}\right)^{\frac{\ell - 1}{2}}, 
    \end{align*}
    where the last inequality follows from the fact that each type D tuple involves at most $\frac{\ell - 1}{2}$ distinct primes in $\mathcal{P}_N$. The desired bound now follows from Mertens' estimate. 
\end{proof}

\subsection{Type B terms}

The treatment for type B terms follows the same approach as our treatment for type D terms above. We begin by proving the analogue of Lemma~\ref{type D lem} for type B terms. This is the only place where we use the irrationality condition $\alpha_i/\alpha_j\notin\mathbb{Q}$ for $i\neq j$. 

\begin{lem} \label{type B lem}
    For any $p\in\mathcal{P}_N$, and any $\gamma'_1\neq \gamma'_2\in \{\gamma_1, \dots, \gamma_k\}$, 
    \begin{equation*}
        \sup_{r\in\mathcal{R}_p}\sum_{\substack{m_1, m_2\in p^{-1}\mathbb{Z}\\ m_1\gamma'_1 + m_2\gamma'_2\in\mathbb{Z} + r}}\frac{\phi(m_1)}{p}\frac{\phi(m_2)}{p}\lll p^{-1}, 
    \end{equation*}
    where the convergence is uniform across $p\in\mathcal{P}_N$. 
\end{lem}

\begin{proof}
    Let $\gamma'_1/\gamma'_2 = g_1/g_2$, where $g_1, g_2$ are coprime positive integers. Let $\gamma'_1/g_1 = \gamma'_2/g_2 = g_3/g_4$, where $g_3, g_4$ are coprime positive integers. As the numerators and denominators of $\gamma'_1, \gamma'_2$ are coprime to $p\in \mathcal{P}_N$, $p\nmid g_1, g_2, g_3, g_4$. Since $g_1/g_2 = \gamma'_1/\gamma'_2\rightarrow\alpha_1/\alpha_2\notin\mathbb{Q}$ as $N\rightarrow\infty$, we have $g_4\gg g_1\asymp g_2\rightarrow\infty$ as $N\rightarrow\infty$. We seek to show that
    \begin{equation*}
    \sup_{r\in\mathcal{R}_p}\sum_{\substack{m_1, m_2\in p^{-1}\mathbb{Z}\\ g_1m_1 + g_2m_2\in g_4/g_3\mathbb{Z} + rg_4/g_3}}\phi(m_1) \phi(m_2)\lll p. 
    \end{equation*}
    For $m_1, m_2\in p^{-1}\mathbb{Z}$, we must have $g_1m_1 + g_2m_2\in p^{-1}\mathbb{Z}$, and any $r\in\mathcal{R}_p$, 
    \begin{equation*}
        (g_4/g_3\mathbb{Z} + rg_4/g_3)\cap \left(p^{-1}\mathbb{Z}\right)\subseteq g_4\mathbb{Z} + r'
    \end{equation*}
    for some $r'\in\mathbb{Z}\cap [0, g_4)$. Therefore, the sum above is upper bounded by
    \begin{equation*}
        \sup_{r'\in\mathbb{Z}\cap [0, g_4)}\sum_{\substack{m_1, m_2\in p^{-1}\mathbb{Z}\\ g_1m_1 + g_2m_2\in g_4\mathbb{Z} + r'}}\phi(m_1) \phi(m_2). 
    \end{equation*}
    We shall make use of the bound that
    \begin{equation*}
        \sum_{m\in x\mathbb{Z} + y}\phi(m)\ll \sum_{m\in x\mathbb{Z} + y}\min(m^{-2}, 1)\ll x^{-1} + 1
    \end{equation*}
    for any real numbers $x > 0$ and $y$. For any integer $s\neq 0, -1$, if $g_1m_1 + g_2m_2 = g_4s + r'$, then $\max(|m_1|, |m_2|)\geq \frac{|g_4s + r'|}{g_1 + g_2}\geq \frac{g_4|s|/2}{g_1 + g_2}$. Therefore, 
    \begin{align*}
        &\sum_{\substack{m_1, m_2\in p^{-1}\mathbb{Z}\\ g_1m_1 + g_2m_2 = g_4s + r'}}\phi(m_1) \phi(m_2)\\
        &\ll \sum_{\substack{m_1, m_2\in p^{-1}\mathbb{Z}\\ g_1m_1 + g_2m_2 = g_4s + r'}}\min(m_1^{-2}, 1) \min(m_2^{-2}, 1)\\
        &= \sum_{\substack{m_1, m_2\in p^{-1}\mathbb{Z}\\ g_1m_1 + g_2m_2 = g_4s + r'}}\min(\max(m_1^{-2}, m_2^{-2}), 1)\min(\min(m_1^{-2}, m_2^{-2}), 1)\\
        &\leq \sum_{\substack{m_1, m_2\in p^{-1}\mathbb{Z}\\ g_1m_1 + g_2m_2 = g_4s + r'}}(\min(m_1^{-2}, 1) + \min(m_2^{-2}, 1))\max(|m_1|, |m_2|)^{-2}\\
        &\ll \left(\frac{g_4|s|}{g_1}\right)^{-2}\sum_{\substack{m_1, m_2\in p^{-1}\mathbb{Z}\\ g_1m_1 + g_2m_2 = g_4s + r'}}\left(\min\left(m_1^{-2}, 1\right) + \min\left(m_2^{-2}, 1\right)\right)\\
        &\ll \left(\frac{g_4|s|}{g_1}\right)^{-2}\left(\frac{p}{g_1} + 1\right). 
    \end{align*}
    Therefore, for any choice of $r'\in\mathbb{Z}\cap [0, g_4)$, we have
    \begin{align*}
        &\sum_{\substack{m_1, m_2\in p^{-1}\mathbb{Z}\\ g_1m_1 + g_2m_2\in g_4\mathbb{Z} + r'}}\phi(m_1) \phi(m_2)\\
        &= \sum_{s = -1, 0}\sum_{\substack{m_1, m_2\in p^{-1}\mathbb{Z}\\ g_1m_1 + g_2m_2\in g_4s + r'}}\phi(m_1) \phi(m_2) + \sum_{s \neq -1, 0}\sum_{\substack{m_1, m_2\in p^{-1}\mathbb{Z}\\ g_1m_1 + g_2m_2\in g_4s + r'}}\phi(m_1) \phi(m_2)\\
        &\ll \sum_{s = -1, 0}\sum_{\substack{m_1, m_2\in p^{-1}\mathbb{Z}\\ g_1m_1 + g_2m_2\in g_4s + r'}}\phi(m_1) + \sum_{s\neq -1, 0}\left(\frac{g_4|s|}{g_1}\right)^{-2}\left(\frac{p}{g_1} + 1\right)\\
        &\ll \frac{p}{g_1} + \left(\frac{g_4}{g_1}\right)^{-2}\left(\frac{p}{g_1} + 1\right)\\
        &\ll \frac{p}{g_1} + 1. 
    \end{align*}
    where the implied constant is absolute and, in particular, independent of $r'$. Therefore, 
    \begin{equation*}
        \sup_{r'\in\mathbb{Z}\cap [0, g_4)}\sum_{\substack{m_1, m_2\in p^{-1}\mathbb{Z}\\ g_1m_1 + g_2m_2\in g_4\mathbb{Z} + r'}}\phi(m_1) \phi(m_2) \ll \frac{p}{g_1} + 1\lll p
    \end{equation*}
    as $N\rightarrow\infty$, since $p\geq \log N\ggg 1$ for $p\in\mathcal{P}_N$. 
\end{proof}

We now prove the analogue of Corollary~\ref{individual type D} for type B terms. 

\begin{cor} \label{individual type B}
    For any type B tuple $(p_{ij})$, 
    \begin{equation*}
        E'\lll \prod_{ij}'p_{ij}^{-1}, 
    \end{equation*}
    where $\prod_{ij}'$ denotes the product over distinct $p_{ij}$'s, and the convergence is uniform across type B tuples (with $\ell$ entries). 
\end{cor}

\begin{proof}
    Let $p = p_{i_1j_1} = p_{i_2j_2}$ where $i_1\neq i_2$. For a prime $q$ in the tuple $(p_{ij})$, let $\mathcal{I}_q := \{(i, j)|i\in [k], j\in [\ell_i], p_{ij} = q\}$ and let $\mathcal{I} := \{(i, j)|i\in [k], j\in [\ell_i], p_{ij}\neq p\} = \cup_{q\neq p} \mathcal{I}_q$. 
    
    For any tuple $(m_{ij})\in \prod (p_{ij}^{-1}\mathbb{Z})$ satisfying $\sum_{ij}m_{ij}\gamma_i\in\mathbb{Z}$ and any prime $q$ in the tuple $(p_{ij})$, note that $\sum_{ij\in \mathcal{I}_q}m_{ij}\gamma_i\in \mathcal{R}_q$. Therefore, we have, by Lemmas~\ref{type B lem} and~\ref{type D lem}, 
    \begin{align*}
        E' &= \sum_{\substack{(m_{ij})\in \prod (p_{ij}^{-1}\mathbb{Z})\\ \sum_{ij}m_{ij}\gamma_i\in\mathbb{Z}}}\left(\prod_{ij \text{ s.t. }p_{ij} \neq p}\frac{\phi(m_{ij})}{p_{ij}}\right)\frac{\phi(m_{i_1j_1})}{p}\frac{\phi(m_{i_2j_2})}{p}\\
        &= \sum_{\substack{(m_{ij})_{ij\in \mathcal{I}}\in \prod_{ij\in \mathcal{I}}(p_{ij}^{-1}\mathbb{Z})\\ \forall q\neq p,  \sum_{ij\in \mathcal{I}_q}m_{ij}\in \mathcal{R}_q}}\left(\left(\prod_{ij\in \mathcal{I}}\frac{\phi(m_{ij})}{p_{ij}}\right)\sum_{\substack{m_{i_1j_1}, m_{i_2j_2}\in p^{-1}\mathbb{Z}\\\gamma_{i_1}m_{i_1j_1} + \gamma_{i_2}m_{i_2j_2}\in \\\mathbb{Z} - \sum_{ij\in\mathcal{I}}m_{ij}\gamma_i}}\frac{\phi(m_{i_1j_1})}{p}\frac{\phi(m_{i_2j_2})}{p}\right)\\
        &\leq \left(\prod_{q\neq p}\sum_{\substack{(m_{ij})_{ij\in \mathcal{I}_q}\in \prod_{ij\in \mathcal{I}_q}(p_{ij}^{-1}\mathbb{Z})\\ \sum_{ij\in \mathcal{I}_q}m_{ij}\in \mathcal{R}_q}}\prod_{ij \in \mathcal{I}_q}\frac{\phi(m_{ij})}{q}\right)\cdot\left(\sup_{r\in\mathcal{R}_p}\sum_{\substack{m_{i_1j_1}, m_{i_2j_2}\in p^{-1}\mathbb{Z}\\\gamma_{i_1}m_{i_1j_1} + \gamma_{i_2}m_{i_2j_2}\in \\\mathbb{Z} + r}}\frac{\phi(m_{i_1j_1})}{p}\frac{\phi(m_{i_2j_2})}{p}\right)\\
        &\lll \left(\prod_{q\neq p}q^{-1}\right)\cdot p^{-1}\\
        & = \prod_{ij}'p_{ij}^{-1}, 
    \end{align*}
    where we take the product over distinct primes $q\neq p$ in the tuple $(p_{ij})$. 
\end{proof}

Lastly, in analogy to Corollary~\ref{type D sum}, we have the following corollary on the total contribution from type B terms, which follows from the argument of Corollary~\ref{type D sum}. 

\begin{cor} \label{type B sum}
    We have
    \begin{equation*}
        \left|\sum_{\text{type B }(p_{ij})}E(p_{ij})\right|\lll \left(\log\log N\right)^{\frac{\ell}{2}}. 
    \end{equation*}
\end{cor}

\section{Type A and C tuples} \label{type A and C}

We now carry out the estimates for type A and C terms, in an analogous manner to the treatment of type B and D terms in Section~\ref{type B and D}. The main difference to Section~\ref{type B and D} stems from the use of a different approximant of $1_{[0, 1)}$, which requires a more technical treatment, given below in Subsection~\ref{indicator approximation 2}. 

\subsection{Smoothing the indicator} \label{indicator approximation 2}

Recall that
\begin{equation*}
    1_{p\mid \floor{x}} = \sum_{m\in \mathbb{Z}}1_{[0, 1)}(x - pm), 
\end{equation*}
may be written as a sum of indicator functions of unit intervals. 

\begin{df}
    Let $\epsilon =\epsilon(N) := J^{-1/2}\geq N^{-o(1)}$. Let 
    \begin{equation*}
        \chi_+ := \epsilon^{-1}1_{[0, 1]}*1_{[0, \epsilon]},\qquad\chi_- := \epsilon^{-1}1_{[0, 1]}*1_{[-\epsilon, 0]}
    \end{equation*}
    approximate $1_{[0, 1)}$. Let 
    \begin{equation*}
        \chi_{p, \pm}(x) := \sum_{m\in \mathbb{Z}}\chi_\pm(x - pm), 
    \end{equation*}
    be the corresponding approximants of $1_{p\mid \floor{x}}$. 
\end{df}

Approximating $1_{p\mid \floor{x}}$ by $\chi_{p, \pm}(x)$ incurs an error for $x\in [0, \epsilon)\mod 1$ and $x\in ( - \epsilon, 0)\mod 1$, respectively. Unfortunately, neither approximant suffices universally. 

\begin{df} \label{partition}
    \footnote{This partitioning is unnecessary in the special case $\beta_1 = \dots = \beta_k = 0$, where one may treat the entire interval $[N]$ at once.}We partition $[N]$ into $N^{1/2}$ intervals $S_1, \dots, S_{N^{1/2}}$ of $N^{1/2}$ integers each. 
\end{df}

We shall choose one of the approximants $\chi_{p, \pm}$ for each interval. Strictly speaking, for Definition~\ref{partition}, we need to assume that $N$ is a square. This is not problematic as it suffices to show Theorem~\ref{main} along square values of $N$. We may also resolve this issue by suitably rounding the lengths of $S_u$. 

\begin{df} \label{igood}
    Call an interval $S$ of integers \emph{$i_+$-good} if 
    \begin{equation*}
        \mathbb{P}_{n\sim U(S)}(\alpha_i n + \beta_i\in [0, \epsilon)\mod 1)\leq \epsilon^{1/4}
    \end{equation*}
    Call $S$ \emph{$i_-$-good} if 
    \begin{equation*}
        \mathbb{P}_{n\sim U(S)}(\alpha_i n + \beta_i\in ( - \epsilon, 0)\mod 1)\leq \epsilon^{1/4}. 
    \end{equation*}
    Call $S$ \emph{$i$-good} if it is either $i_+$-good or $i_-$-good, and call $S$ \emph{$i$-bad} otherwise. Call $S$ \emph{good} if it is $i$-good for all $i \in [k]$, and call $S$ \emph{bad} otherwise. 
\end{df}

We shall approximate $1_{p\mid \floor{\alpha_i n + \beta_i}}$ by $\chi_{p, +}(\alpha_i n + \beta_i)$ on $i_+$-good intervals and $\chi_{p, -}(\alpha_i n + \beta_i)$ on $i_-$-good intervals. We first upper bound the number of bad intervals amongst $S_1, \dots, S_{N^{1/2}}$. 

To do so, we split into cases according to how well $\alpha_i$ approximates rationals. 

\begin{df}
    Let 
    \begin{equation*}
        \mathfrak{M} = \mathfrak{M}_N := \bigcup_{a\in\mathbb{Z}, b\in [J]}[a/b - N^{-1/3}, a/b + N^{-1/3}], 
    \end{equation*}
    denote the set of \emph{major arcs} and let 
    \begin{equation*}
        \mathfrak{m} = \mathfrak{m}_N := \mathbb{R}\backslash\mathfrak{M}
    \end{equation*}
    denote the set of \emph{minor arcs}. 
\end{df}

Since $J \leq N^{o(1)}$, the major arcs are non-intersecting (for sufficiently large $N$). 

\begin{lem} \label{major arcs}
    If $\alpha_i\in \mathfrak{M}$, then there are at most 
    \begin{equation*}
        O(\epsilon^{1/4}N^{1/2})
    \end{equation*}
    $i$-bad intervals amongst $S_1, \dots, S_{N^{1/2}}$. 
\end{lem}

\begin{proof}
    Recall that we assume throughout that $N$ is sufficiently large in an absolute sense, and $\epsilon = \epsilon(N)\rightarrow 0$ as $N\rightarrow\infty$. Therefore, we may implicitly assume that $\epsilon$ is sufficiently small in an absolute sense. 

    Let $\alpha_i = a/b + \delta$, where $a$ and $b$ are coprime, $|\delta|\leq N^{-1/3}$, and $b\leq J = \epsilon^{-2}$. 

    We first treat the case where $b\geq \epsilon^{-1/3}$. We claim that over any interval $I$ of $b$ integers, 
    \begin{equation*}
        \mathbb{P}_{n\sim U(I)}(\alpha_i n + \beta_i\in [0, \epsilon)\mod 1)\ll \epsilon + b^{-1}. 
    \end{equation*}
    It follows from this claim that any $S_u$ is $i_+$ good. We now prove this claim. Let $n_0$ denote the least element of $I$. For $n\in I$, we have
    \begin{equation*}
        \alpha_i n + \beta_i \equiv an/b + (\delta n_0 + \beta) + \delta(n - n_0)\mod 1. 
    \end{equation*}
    Note that $|\delta(n - n_0)|\leq N^{-1/3}J\leq \epsilon$ when $N$ is sufficiently large. Therefore, it suffices to show that
    \begin{equation*}
        \mathbb{P}_{n\sim U(I)}(an/b + (\delta n_0 + \beta)\in [-\epsilon, 2\epsilon]\mod 1)\ll \epsilon + b^{-1}. 
    \end{equation*}
    This follows since $an/b\sim U(b^{-1}\mathbb{Z})$ modulo 1. 

    We now consider the case where $b < \epsilon^{-1/3}$. We partition each interval $S_u$ of integers into $b$ sets $S_u^{(r)}$ by their residue $r = 0, \dots, b - 1$ modulo $b$. We split into three cases based on the size of $|\delta|$. 
    \begin{itemize}
        \item When $|\delta|\geq \epsilon^{2/3}N^{-1/2}$, we show that in fact any $S_u$ is $i_+$-good. It suffices to show for each residue $r\mod b$ that 
        \begin{equation*}
            \mathbb{P}_{n\sim U(S_u^{(r)})}(\alpha_i n + \beta_i\in [0, \epsilon)\mod 1)\ll \epsilon^{1/4}
        \end{equation*}
        Note however, for $n\in S_u^{(r)}$
        \begin{equation*}
            \alpha_i n + \beta_i\equiv \delta n + ar/b + \beta_i\mod 1
        \end{equation*}
        Therefore, the set of $n\in S_u^{(r)}$ for which 
        \begin{equation*}
            \alpha_i n + \beta_i\in [0, \epsilon)\mod 1
        \end{equation*}
        may be partitioned into at most $\ceil{|\delta| N^{1/2}} + 1$ arithmetic progressions of common difference $b$ each of length at most $\ceil{\epsilon/(b|\delta|)}$. Therefore, 
        \begin{equation*}
            \mathbb{P}_{n\sim U(S_u^{(r)})}(\alpha_i n + \beta_i\in [0, \epsilon)\mod 1)\ll \frac{1}{\left|S_u^{(r)}\right|}\left(\ceil{|\delta| N^{1/2}} + 1\right)\cdot \ceil{\frac{\epsilon}{b|\delta|}}. 
        \end{equation*}
        Since $\epsilon/(b|\delta|)\gg 1$ and $|\delta| N^{1/2}\gg \epsilon^{2/3}$, we have
        \begin{align*}
            \ceil{\frac{\epsilon}{b|\delta|}}\ll \frac{\epsilon}{b|\delta|}, &\qquad \ceil{|\delta| N^{1/2}} + 1\ll \epsilon^{-2/3}|\delta| N^{1/2}. 
        \end{align*}
        Hence
        \begin{equation*}
            \mathbb{P}_{n\sim U(S_u^{(r)})}(\alpha_i n + \beta_i\in [0, \epsilon)\mod 1)\ll \frac{b}{N^{1/2}}\cdot \epsilon^{-2/3}|\delta| N^{1/2}\cdot \frac{\epsilon}{b|\delta|} = \epsilon^{1/3}\lll \epsilon^{1/4}. 
        \end{equation*}
        
        \item We now deal with the case $0<|\delta| < \epsilon^{2/3}N^{-1/2}$. In this case, the values taken by 
        \begin{equation*}
            \alpha_i n + \beta_i\equiv \delta n + ar/b + \beta_i\mod 1
        \end{equation*}
        for $n\in S_u^{(r)}$ is contained in the interval $I_u^{(r)}: = [\delta (u - 1) N^{1/2} + ar/b + \beta_i, \delta u N^{1/2} + ar/b + \beta_i]\mod 1$ of length $|\delta| N^{1/2}\leq \epsilon^{2/3}$, where we reverse the endpoints of the interval if $\delta < 0$. We claim that if $S_u$ is $i$-bad then $0\in I_u^{(r)}$ for some $r\mod b$. We first see that this suffices, since $0\in I_u^{(r)}$ for at most $2\ceil{N|\delta|}\ll \epsilon^{2/3}N^{1/2}$ values of $u$ for each $r\mod b$, hence for at most $O(b\epsilon^{2/3}N^{1/2})\ll \epsilon^{1/4}N^{1/2}$ values of $u\in [N^{1/2}]$. 
    
        We now prove this claim. Assume that $0\notin I_u^{(r)}$ for any $r\mod b$. We first note that the intervals $I_u^{(r)}$ are separated by gaps of length $b^{-1} - |\delta| N^{1/2}\geq \epsilon^{1/3} - \epsilon^{2/3}\ggg \epsilon$. Therefore, there is at most one choice of $r_0$ for which $I_u^{(r_0)}$ intersects $[-\epsilon, \epsilon]\mod 1$. Since $0\notin I_u^{(r_0)}$, $I_u^{(r_0)}$ must be disjoint from $[0, \epsilon]$ or from $[-\epsilon, 0]$ modulo 1. Hence
        \begin{equation*}
            \{\alpha_i n + \beta_i|n\in S_u\}\subseteq \bigcup_r I_u^{(r)} \mod 1
        \end{equation*}
        must be disjoint from $[0, \epsilon]$ or from $[-\epsilon, 0]$ modulo 1. Therefore, $S_u$ is $i$-good. 
    
        \item Finally, when $\delta = 0$, we shall show that any $S_u$ is $i$-good. For any $u$, we have
        \begin{equation*}
            \{\alpha_i n + \beta_i|n\in S_u\}\subseteq b^{-1}\mathbb{Z} + \beta_i. 
        \end{equation*}
        Since $b^{-1}\geq \epsilon^{1/3}\ggg \epsilon$, the set $b^{-1}\mathbb{Z} + \beta_i$ is disjoint from at least one of $[0, \epsilon)$ and $(-\epsilon, 0)$ modulo 1. Therefore, $S_u$ is $i$-good. \qedhere
    \end{itemize}
\end{proof}

\begin{lem} \label{minor arcs}
    If $\alpha_i\in \mathfrak{m}$, then all intervals $S_1, \dots, S_{N^{1/2}}$ are $i$-good. 
\end{lem}

\begin{proof}
    We prove that, in fact, any $S_u$ is $i^+$-good. Since $\alpha_i\in\mathfrak{m}$, $\|\alpha_i m\|_\mathbb{Z}\geq mN^{-1/3}$ for any positive integer $m\leq J$. We employ Fourier theoretic techniques on $\mathbb{R}/\mathbb{Z}$ to upper bound the probability
    \begin{equation*}
        \mathbb{P}_{n\sim U(S_u)}\left(\alpha_i n + \beta_i\in [0, \epsilon)\mod 1\right). 
    \end{equation*}
    Let 
    \begin{equation*}
        g_\epsilon(x) = \sum_{m'\in \mathbb{Z}}e^{-\frac{(x - m')^2}{2\epsilon^2}} = \sqrt{2\pi}\cdot\epsilon\sum_{m\in \mathbb{Z}}e(mx)e^{-2\pi^2\epsilon^2 m^2}
    \end{equation*}
    denote a $1$-periodised Gaussian. We have 
    \begin{align*}
        &\mathbb{P}_{n\sim U(S_u)}\left(\alpha_i n + \beta_i\in [0, \epsilon)\mod 1\right)\\
        &\ll N^{-1/2}\sum_{n \in S_u} g_\epsilon(\alpha_i n + \beta)\\
        & = \sqrt{2\pi}\epsilon N^{-1/2}\sum_{m\in\mathbb{Z}}e^{-2\pi^2\epsilon^2 m^2} \sum_{n \in S_u} e(m(\alpha_i n + \beta))\\
        &\ll \epsilon N^{-1/2}\sum_{m\in \mathbb{Z}} e^{-2\pi^2\epsilon^2m^2}\cdot \left|\frac{e^{2\pi i N^{1/2} \alpha_i m} - 1}{(e^{2\pi i \alpha_i m} - 1)}\right|\\
        &\ll \epsilon \sum_{m\in \mathbb{Z}}  \frac{e^{-2\pi^2\epsilon^2m^2}}{\max(N^{1/2}\|\alpha_i m\|_\mathbb{Z}, 1)}\\
        &\ll \epsilon + \epsilon \sum_{0<|m|\leq J}  \frac{e^{-2\pi^2\epsilon^2m^2}}{N^{1/2}\|\alpha_i m\|_\mathbb{Z}} + \epsilon \sum_{|m| > J}  e^{-2\pi^2\epsilon^2m^2}\\
        &\ll \epsilon + N^{-1/6} + e^{-2\pi^2\epsilon^2J^2}\\
        &\ll \epsilon\\
        &\lll \epsilon^{1/4}. \qedhere
    \end{align*}
\end{proof}

We combine Lemma~\ref{major arcs} for the major arcs and Lemma~\ref{minor arcs} for the minor arcs. 

\begin{cor} \label{both arcs}
    There are at most 
    \begin{equation*}
        O(\epsilon^{1/4}N^{1/2})
    \end{equation*}
    $i$-bad intervals amongst $S_1, \dots, S_{N^{1/2}}$. 
\end{cor}

\begin{df}
    Let $G = G_N\subseteq [N^{1/2}]$ denote the set of $u\in [N^{1/2}]$ for which $S_u$ is good, and let $B = B_N\subseteq [N^{1/2}]$ denote the set of $u\in [N^{1/2}]$ for which $S_u$ is bad. 
\end{df}

We decompose $E$ to investigate the contribution from each interval $S_u$. 

\begin{df}
    For $u \in [N^{1/2}]$, let 
    \begin{equation*}
        E_u = E_u(p_{ij}) := \mathbb{E}_{n\sim U[S_u]}\left(\prod_{ij} (1_{p_{ij}\mid \floor{\alpha_i n + \beta_i}} - p_{ij}^{-1})\right). 
    \end{equation*}
\end{df}

We may write
\begin{equation*}
    E = N^{-1/2}\sum_{u = 1}^{N^{1/2}}E_u. 
\end{equation*}

For $u\in G$, we are ready to take a suitable continuous approximant for $1_{p\mid \floor{\alpha_i n + \beta_i}}$ over $n\in S_u$, and, in doing so, produce an approximant of $E_u$. 

\begin{df} \label{E''u}
    For $u\in G$, let
    \begin{equation*}
        E''_u = E''_u(p_{ij}):= \mathbb{E}_{n\sim U(S_u)}\left(\prod_{ij} (\chi_{p_{ij}, \pm}(\alpha_i n + \beta_i) - p_{ij}^{-1})\right), 
    \end{equation*}
    where we take $\chi_{p_{ij}, +}$ if $S_u$ is $i_+$-good and $\chi_{p_{ij}, -}$ otherwise. Note that as $S_u$ is $i$-good, in the latter case $S_u$ must be $i_-$-good. 
\end{df}

\begin{lem} \label{E_u and E''_u}
    For any $u\in G$, we have
    \begin{equation*}
        E_u = E''_u + O\left(\epsilon^{1/4}\right). 
    \end{equation*}
\end{lem}

\begin{proof}
    We claim that for the choice of $\chi_{p_{ij}, \pm}$ in Definition~\ref{E''u}, 
    \begin{equation*}
        \mathbb{P}_{n\sim U[S_u]}\left(\chi_{p_{ij}, \pm}(\alpha_i n + \beta_i)\neq 1_{p_{ij}\mid \floor{\alpha_i n + \beta_i}}\right)\ll \epsilon^{1/4}. 
    \end{equation*}
    This follows from Definition~\ref{igood} of $i_\pm$-good intervals, once we note that we only take the choice $\chi_{p_{ij}, +}$ when $S_u$ is $i_+$-good, and the choice $\chi_{p_{ij}, -}$ when $S_u$ is $i_-$-good. Therefore, 
    \begin{equation*}
        \mathbb{P}_{n\sim U[S_u]}\left(\prod_{ij}(\chi_{p_{ij}, \pm}(\alpha_i n + \beta_i) - p_{ij}^{-1})\neq \prod_{ij}(1_{p_{ij}\mid \floor{\alpha_i n + \beta_i}} - p_{ij}^{-1})\right)\ll \ell \epsilon^{1/4}\ll \epsilon^{1/4}. 
    \end{equation*}
    Since $\prod_{ij}(\chi_{p_{ij}, \pm}(\alpha_i n + \beta_i) - p_{ij}^{-1})\in [-1, 1]$ and $\prod_{ij}(1_{p_{ij}\mid \floor{\alpha_i n + \beta_i}} - p_{ij}^{-1})\in [-1, 1]$, we have
    \begin{equation*}
        \left|\prod_{ij}(\chi_{p_{ij}, \pm}(\alpha_i n + \beta_i) - p_{ij}^{-1}) - \prod_{ij}(1_{p_{ij}\mid \floor{\alpha_i n + \beta_i}} - p_{ij}^{-1})\right|\leq 2. 
    \end{equation*}
    Therefore, 
    \begin{align*}
        |E_u - E''_u|&\leq \mathbb{E}_{n\sim U(S_u)}\left(\left|\prod_{ij}(\chi_{p_{ij}, \pm}(\alpha_i n + \beta_i) - p_{ij}^{-1}) - \prod_{ij}(1_{p_{ij}\mid \floor{\alpha_i n + \beta_i}} - p_{ij}^{-1})\right|\right)\\
        &\leq 2\cdot \mathbb{P}_{n\sim U[S_u]}\left(\prod_{ij}(\chi_{p_{ij}, \pm}(\alpha_i n + \beta_i) - p_{ij}^{-1})\neq \prod_{ij}(1_{p_{ij}\mid \floor{\alpha_i n + \beta_i}} - p_{ij}^{-1})\right)\\
        &\ll\epsilon^{1/4}. \qedhere
    \end{align*}
\end{proof}

Combining the approximants $E''_u$ of $E_u$ for $u\in G$, we recover an approximant of $E$. 

\begin{df}
    Let
    \begin{equation*}
        E'' = E''(p_{ij}) := N^{-1/2}\sum_{u\in G}E_u''. 
    \end{equation*}
\end{df}

\begin{lem} \label{E and E''}
    We have
    \begin{equation*}
        E = E'' + O(\epsilon^{1/4}). 
    \end{equation*}
\end{lem}

\begin{proof}
    We have
    \begin{align*}
        |E - E''| \leq N^{-1/2}\sum_{u\in G}|E_u - E''_u| + N^{-1/2}\sum_{u\in B}|E_u|. 
    \end{align*}
    By Lemma~\ref{E_u and E''_u}, $N^{-1/2}\sum_{u\in G}|E_u - E''_u|\ll N^{-1/2}|G|\epsilon^{1/4}\leq \epsilon^{1/4}$. Therefore, it suffices to show that 
    \begin{equation*}
        N^{-1/2}\sum_{u\in B}|E_u|\ll \epsilon^{1/4}. 
    \end{equation*}
    Since $\prod_{ij} (1_{p_{ij}\mid \floor{\alpha_i n + \beta_i}} - p_{ij}^{-1})\in [-1, 1]$, we have $|E_u|\leq 1$. Combining this observation with Corollary~\ref{both arcs}, which asserts that $|B|\ll \epsilon^{1/4}N^{1/2}$, we have
    \begin{equation*}
        N^{-1/2}\sum_{u\in B}|E_u|\leq N^{-1/2}|B|\ll \epsilon^{1/4}, 
    \end{equation*}
    as desired. 
\end{proof}

\subsection{The Fourier transform} \label{Fourier}

We now evaluate $E''_u$ by taking a suitable Fourier transform. 

\begin{df}
    Let 
    \begin{equation*}
        \phi_\pm(y) = \phi_{\epsilon, \pm}(y) := \widehat{\chi_{\epsilon, \pm}}(y) = \int \chi_{\epsilon, \pm}(x) e(-xy)dx, 
    \end{equation*}
    and let 
    \begin{equation*}
        \theta_u(x) := N^{-1/2}\sum_{n\in S_u}e(nx). 
    \end{equation*}
\end{df}

Note that 
\begin{equation*}
|\phi_+(m)| = |\phi_-(m)| \ll\min(1, m^{-1})\min(1, \epsilon^{-1}m^{-1}), 
\end{equation*}
and therefore, 
\begin{align}
    \sum_{m\in p^{-1}\mathbb{Z}}\frac{1}{p}|\phi_\pm(m)|&\ll \log(\epsilon^{-1}),  \label{Fourier l1 estimate all}\\
    \sum_{m\in p^{-1}\mathbb{Z}, |m|\geq J}\frac{1}{p}|\phi_\pm (m)|&\ll \epsilon^{-1}J^{-1}. \label{Fourier l1 estimate tail}
\end{align}

Note that the restrictions of $\frac{1}{p}\phi_\pm$ on $p^{-1}\mathbb{Z}$ form the Fourier coefficients of the $p$-periodisations $\chi_{p, \pm}$ of $\chi_\pm$. Since 
\begin{equation*}
    \sum_{m\in p^{-1}\mathbb{Z}}\frac{1}{p}|\phi_\pm(m)| < \infty
\end{equation*}
is absolutely convergent, we have the Fourier inversion formulae
\begin{equation} \label{Fourier inversion}
    \sum_{m\in p^{-1}\mathbb{Z}}\frac{\phi_\pm(m)}{p}e(mx) = \chi_{p, \pm}(x). 
\end{equation}

\begin{lem}
    For any $u\in G$, 
    \begin{equation*}
        E''_u = \sum_{m_{ij}\neq 0\in p_{ij}^{-1}\mathbb{Z}}\theta_u\left(\sum_{ij} \alpha_im_{ij}\right)\prod_{ij}\frac{\phi_\pm(m_{ij})e\left(\beta_im_{ij}\right)}{p_{ij}}, 
    \end{equation*}
    where the sign on the right-hand side is taken in accordance with the choice of $\chi_{p_{ij}, \pm}$ in the definition of $E''_u$ (Definition~\ref{E''u}). 
\end{lem}

\begin{proof}
    Starting on the right-hand side, we have, by Fubini's theorem and (\ref{Fourier inversion}), 
    \begin{align*}
        &\sum_{m_{ij}\neq 0\in p_{ij}^{-1}\mathbb{Z}}\theta_u\left(\sum_{ij} \alpha_im_{ij}\right)\prod_{ij}\frac{\phi_\pm(m_{ij})e\left(\beta_im_{ij}\right)}{p_{ij}}\\
        &= N^{-1/2}\sum_{m_{ij}\neq 0\in p_{ij}^{-1}\mathbb{Z}}\sum_{n\in S_u} e\left(n\sum_{ij} \alpha_im_{ij}\right)\prod_{ij}\frac{\phi_\pm(m_{ij})e\left(\beta_im_{ij}\right)}{p_{ij}}\\
        &= N^{-1/2}\sum_{n\in S_u} \prod_{ij}\sum_{m_{ij}\neq 0\in p_{ij}^{-1}\mathbb{Z}}\frac{\phi_\pm(m_{ij})}{p_{ij}}e\left((\alpha_in + \beta_i)m_{ij}\right)\\
        &= N^{-1/2}\sum_{n\in S_u} \prod_{ij}(\chi_{p_{ij}, i}(\alpha_i n + \beta_i) - p_{ij}^{-1})\\
        &= E''_u. \qedhere
    \end{align*}
\end{proof}

We note that
\begin{equation*}
    |\theta_u(x)| = N^{-1/2}\left|\frac{e(N^{1/2}x) - 1}{e(x) - 1}\right|\ll \frac{1}{\max(N^{1/2}\|x\|_\mathbb{Z}, 1)}. 
\end{equation*}
Therefore, for any $u\in G$, 
\begin{equation*}
    \left|E''_u\right|\ll \sum_{m_{ij}\neq 0\in p_{ij}^{-1}\mathbb{Z}}\frac{1}{\max(N^{1/2}\|\sum_{ij} \alpha_im_{ij}\|_\mathbb{Z}, 1)}\prod_{ij}\frac{|\phi_\pm(m_{ij})|}{p_{ij}}, 
\end{equation*}
Substituting this into the definition of $E''$, we arrive at the following corollary. 

\begin{cor} \label{E'' cor}
    \begin{equation*}
        |E''| \ll \sum_{m_{ij}\neq 0\in p_{ij}^{-1}\mathbb{Z}}\frac{1}{\max(N^{1/2}\|\sum_{ij} \alpha_im_{ij}\|_\mathbb{Z}, 1)}\prod_{ij}\frac{|\phi_\pm(m_{ij})|}{p_{ij}}. 
    \end{equation*}
\end{cor}

We now isolate the main component of the right-hand side. 

\begin{df}
    Let 
    \begin{equation*}
        E''' = E'''(p_{ij}):= \sum_{\substack{(m_{ij})\in \prod (p_{ij}^{-1}\mathbb{Z}\backslash\mathbb{Z})\\ \sum_{ij}m_{ij}\gamma_i\in\mathbb{Z}}}\prod_{ij}\frac{|\phi_\pm(m_{ij})|}{p_{ij}}. 
    \end{equation*}
\end{df}

\begin{cor} \label{E'' and E'''}
    We have
    \begin{equation*}
        |E''| \ll E''' + J^{-1/3}. 
    \end{equation*}
\end{cor}

\begin{proof}
    We first note that $\phi_\pm$ vanishes on $\mathbb{Z}\backslash\{0\}$, therefore, we may restrict the sum in Corollary~\ref{E'' cor} to $m_{ij}\in p_{ij}^{-1}\mathbb{Z}\backslash\mathbb{Z}$. We first deal with the case when some $|m_{ij}|\geq J$. Here, by (\ref{Fourier l1 estimate all}) and (\ref{Fourier l1 estimate tail}), we have
    \begin{equation*}
        \sum_{\substack{(m_{ij})\in \prod (p_{ij}^{-1}\mathbb{Z}\backslash\mathbb{Z})\\\text{some }|m_{ij}|\geq J}}\prod_{ij}\frac{|\phi_\pm(m_{ij})|}{p_{ij}}\ll \ell (\epsilon^{-1}J^{-1})\log (\epsilon^{-1})^{\ell - 1}\ll J^{-1/3}, 
    \end{equation*}
    where we note that $\epsilon = J^{-1/2}$. Therefore, it suffices to show that
    \begin{equation*}
        \sum_{\substack{(m_{ij})\in \prod (p_{ij}^{-1}\mathbb{Z}\cap (-J, J)\backslash\mathbb{Z})\\ \sum_{ij}m_{ij}\gamma_i \notin\mathbb{Z}}}\frac{1}{\max(N^{1/2}\|\sum_{ij} \alpha_im_{ij}\|_\mathbb{Z}, 1)}\prod_{ij}\frac{|\phi_\pm(m_{ij})|}{p_{ij}}\ll J^{-1/3}. 
    \end{equation*}
    Assume that $N$ is sufficiently large so that $L\geq \ell$. By Lemma~\ref{key lin alg}, for any tuple $(m_{ij})$ in the preceding sum, $\|\sum_{ij}\alpha_im_{ij}\|_\mathbb{Z}\geq N^{-1/4}$. Therefore, by (\ref{Fourier l1 estimate all}), 
    \begin{align*}
        &\sum_{\substack{(m_{ij})\in \prod (p_{ij}^{-1}\mathbb{Z}\cap (-J, J)\backslash\mathbb{Z})\\ \sum_{ij}m_{ij}\gamma_i \notin\mathbb{Z}}}\frac{1}{\max(N^{1/2}\|\sum_{ij} \alpha_im_{ij}\|_\mathbb{Z}, 1)}\prod_{ij}\frac{|\phi_\pm(m_{ij})|}{p_{ij}}\\
        &\ll N^{-1/4}\sum_{\substack{(m_{ij})\in \prod (p_{ij}^{-1}\mathbb{Z}\cap (-J, J)\backslash\mathbb{Z})\\ \sum_{ij}m_{ij}\gamma_i \notin\mathbb{Z}}}\prod_{ij}\frac{|\phi_\pm(m_{ij})|}{p_{ij}}\\
        &\ll \frac{\log(\epsilon^{-1})^\ell}{N^{1/4}}, 
    \end{align*}
    which suffices since $\epsilon^{-1}, J\leq N^{o(1)}$. 
\end{proof}

Combining Lemma~\ref{E and E''} and Corollary~\ref{E'' and E'''}, we arrive at the following upper bound for $|E|$. 

\begin{cor}
    We have
    \begin{equation*}
        |E| \ll E''' + J^{-1/8}. 
    \end{equation*}
\end{cor}

\subsection{Type C terms}

\begin{lem}
    For any type C tuple $(p_{ij})$, $E'''(p_{ij}) = 0$. 
\end{lem}

\begin{proof}
    The sum defining $E'''$ sums over an empty set for type C tuples. In particular, if $p_{ij}$ is distinct from all other primes in the tuple, there is no valid choice of the tuple $(m_{ij})$ in this sum. 
\end{proof}

\begin{cor} \label{individual type C}
    For a type C tuple $(p_{ij})$, 
    \begin{equation*}
        |E(p_{ij})|\ll J^{-1/8}. 
    \end{equation*}
\end{cor}

\begin{cor} \label{type C sum}
    We have
    \begin{equation*}
        \left|\sum_{\text{type C }(p_{ij})}E(p_{ij})\right|\ll R^\ell J^{-1/8}\ll 1. 
    \end{equation*}
\end{cor}

\subsection{Type A terms}

For a type A tuple $(p_{ij})$, we shall in fact reduce $E(p_{ij})$ to a sum involving $E(p_{ij}')$ for type C tuples $(p_{ij}')$ obtained by removing some primes from the tuple $(p_{ij})$. Note that these type C tuples have $\ell'\leq \ell/2$ entries. 

Strictly speaking, our previous bounds for such type C tuples may implicitly depend on $\ell'$. However, as $\ell'\leq \ell/2$, all implied constants may be uniformed bounded across choices of $\ell'$ given $\ell$. 

\begin{lem} \label{individual type A}
    For any type A tuple $(p_{ij})$, 
    \begin{equation*}
        E(p_{ij}) = \prod_{ij}' \left(\frac{1}{p_{ij}} - \frac{1}{p_{ij}^2}\right) + O\left(J^{-1/8}\right), 
    \end{equation*}
    where $\prod_{ij}'$ denotes the product over distinct $p_{ij}$'s. 
\end{lem}

\begin{proof}
    We have
    \begin{align*}
        E(p_{ij}) &= \mathbb{E}_{n\sim U[N]}\left(\prod_{ij}' (1_{p_{ij}\mid\floor{\alpha_i n + \beta_i}} - p_{ij}^{-1})^2\right)\\
        &= \mathbb{E}_{n\sim U[N]}\left(\prod_{ij}' (1_{p_{ij}\mid\floor{\alpha_i n + \beta_i}} - 2p_{ij}^{-1}1_{p_{ij}\mid\floor{\alpha_i n + \beta_i}} + p_{ij}^{-2})\right)\\
        &= \mathbb{E}_{n\sim U[N]}\left(\prod_{ij}' \left((1 - 2p_{ij}^{-1})(1_{p_{ij}\mid\floor{\alpha_i n + \beta_i}} - p_{ij}^{-1}) + \left(p_{ij}^{-1} - p_{ij}^{-2}\right)\right)\right)\\
        &= \prod_{ij}' \left(\frac{1}{p_{ij}} - \frac{1}{p_{ij}^2}\right) + O\left(J^{-1/8}\right), 
    \end{align*}
    since all terms other than $\prod_{ij}' \left(p_{ij}^{-1} - p_{ij}^{-2}\right)$ in the binomial expansion of 
    \begin{equation*}
        \mathbb{E}_{n\sim U[N]}\left(\prod_{ij}' \left(\underbrace{(1 - 2p_{ij}^{-1})(1_{p_{ij}\mid\floor{\alpha_i n + \beta_i}} - p_{ij}^{-1})} + \underbrace{(p_{ij}^{-1} - p_{ij}^{-2})}\right)\right)
    \end{equation*}
    is a type C term bounded by Corollary~\ref{individual type C}. 
\end{proof}

\begin{cor} \label{type A sum}
    We have
    \begin{equation*}
        \sum_{\text{type A }(p_{ij})}E(p_{ij}) = (C_{\ell_1, \dots, \ell_k} + o(1)) \left(\log\log N\right)^{\frac{\ell}{2}}, 
    \end{equation*}
    where $C_{\ell_1, \dots, \ell_k} = 0$ if some $\ell_i$ is odd, and $C_{\ell_1, \dots, \ell_k} = \prod_{i = 1}^k\frac{\ell_i!}{2^{\ell_i/2}(\ell_i/2)!}$ otherwise. 
\end{cor}

\begin{proof}
    If $\ell_i$ is odd for some $i$, then there are no type A tuples, and the desired equality holds trivially. We assume now that $\ell_1, \dots, \ell_k$ are all even. We have
    \begin{align*}
        &\sum_{\text{type A }(p_{ij})}E(p_{ij})\\
        &= \sum_{\text{type A }(p_{ij})}\prod_{ij}' (p_{ij}^{-1} - p_{ij}^{-2})+ O(R^\ell J^{-1/8})\\
        &= \sum_{p_1 < \dots < p_{\ell/2}\in\mathcal{P}_N} \binom{\ell/2}{\ell_1/2, \dots, \ell_k/2}\prod_{i = 1}^k\frac{\ell_i!}{2^{\ell_i/2}}\prod_{j = 1}^{\ell/2} (p_j^{-1} - p_j^{-2})+ O(1)\\
        &= C_{\ell_1, \dots, \ell_k}\sum_{\text{distinct } p_1, \dots, p_{\ell/2}\in\mathcal{P}_N} \prod_{j = 1}^{\ell/2} (p_j^{-1} - p_j^{-2})+ O(1)\\
        &= C_{\ell_1, \dots, \ell_k}\sum_{p_1, \dots, p_{\ell/2}\in\mathcal{P}_N} \prod_{j = 1}^{\ell/2} (p_j^{-1} - p_j^{-2}) + O\left(1 + \sum_{\substack{\text{not pairwise distinct}\\ p_1, \dots, p_{\ell/2}\in\mathcal{P}_N}} \prod_{j = 1}^{\ell/2} (p_j^{-1} - p_j^{-2})\right)\\
        &= C_{\ell_1, \dots, \ell_k}\left(\sum_{p\in\mathcal{P}_N}(p^{-1} - p^{-2})\right)^{\ell/2} + O\left(1 + \left(\sum_{p\in\mathcal{P}_N}(p^{-1} - p^{-2})\right)^{\frac{\ell}{2} - 2}\left(\sum_{p\in\mathcal{P}_N}p^{-2}\right)\right)\\
        &= (C_{\ell_1, \dots, \ell_k} + o(1))\left(\log\log N\right)^{\ell/2}, 
    \end{align*}
    by (\ref{good primes estimate}). 
\end{proof}

\section{Estimating the mixed moment} \label{assembly}

We now put everything together to estimate the mixed moment
\begin{equation*}
    \mathbb{E}_{n\sim U[N]}\left(\prod_{i = 1}^k \left(\frac{\widetilde{\omega'_N}(\floor{\alpha_i n + \beta_i})}{\sqrt{\log\log N}}\right)^{\ell_i}\right), 
\end{equation*}
and establish Proposition~\ref{reduced main}, which in turn completes the proof of Theorem~\ref{main}. 

\begin{proof}[Proof of Proposition~\ref{reduced main}]
    By (\ref{moment decomposition}), we may decompose the mixed moment as
    \begin{align*}
        &\mathbb{E}_{n\sim U[N]}\left(\prod_{i = 1}^k \left(\frac{\widetilde{\omega'_N}(\floor{\alpha_i n + \beta_i})}{\sqrt{\log\log N}}\right)^{\ell_i}\right)\\
        &= \frac{1}{(\log\log N)^{\ell/2}}\sum_{p_{ij}\in \mathcal{P}_N}E(p_{ij})\\
        &= \frac{1}{(\log\log N)^{\ell/2}}\left(\sum_{\text{type A }(p_{ij})}E(p_{ij}) + \sum_{\text{type B }(p_{ij})}E(p_{ij}) + \sum_{\text{type C }(p_{ij})}E(p_{ij}) + \sum_{\text{type D }(p_{ij})}E(p_{ij})\right)\\
        &= \frac{1}{(\log\log N)^{\ell/2}}\left((C_{\ell_1, \dots, \ell_k} + o(1))(\log\log N)^{\ell/2} + o\left((\log\log N)^{\ell/2}\right) + O(1) + O\left((\log\log N)^{\frac{\ell - 1}{2}}\right)\right)\\
        &=C_{\ell_1, \dots, \ell_k} + o(1), 
    \end{align*}
    by Corollaries~\ref{type A sum}, \ref{type B sum}, \ref{type C sum}, and~\ref{type D sum} for type A, B, C, and D terms, respectively. Therefore, 
    \begin{equation*}
        \mathbb{E}_{n\sim U[N]}\left(\prod_{i = 1}^k \left(\frac{\widetilde{\omega'_N}(\floor{\alpha_i n + \beta_i})}{\sqrt{\log\log N}}\right)^{\ell_i}\right)\rightarrow C_{\ell_1, \dots, \ell_k} = \mathbb{E}_{Z\sim\mathcal{N}(0, I_k)}(Z_1^{\ell_1}\cdots Z_k^{\ell_k})
    \end{equation*}
    as $N\rightarrow \infty$. 
\end{proof}

\section{Quantitative convergence} \label{quant conv}

We discuss here quantitative bounds on the rate of convergence of Erd\H{o}s--Kac laws. The three qualitative Erd\H{o}s--Kac laws we discuss, Theorem~\ref{Beatty}, Theorem~\ref{main} and Conjecture~\ref{gen poly conj}, each depend on a set of parameters. Specifically, 
\begin{itemize}
    \item Theorem~\ref{Beatty} concerns $\omega(\floor{\alpha n + \beta})$ for a Beatty sequence $\floor{\alpha n + \beta}$ depending on real parameters $\alpha > 0$ and $\beta$, 
    \item Theorem~\ref{main} concerns the joint distribution of $\omega(\floor{\alpha_i n + \beta_i})$ for multiple Beatty sequences, depending on real parameters $\alpha_i > 0, \beta_i$ for which $\alpha_i/\alpha_j\notin\mathbb{Q}$, 
    \item Conjecture~\ref{gen poly conj} concerns $\omega(\floor{f(n)})$ a generalised polynomial $\floor{f(n)}$, depending on the coefficients of $f$, at least one of which (other than the constant coefficient) must be irrational. 
\end{itemize}
We note that Theorem~\ref{Beatty} does not require the irrationality of any parameter, whilst Theorem~\ref{main} and Conjecture~\ref{gen poly conj} do. In fact, the qualitative Erd\H{o}s--Kac laws asserted by Theorem~\ref{main} and Conjecture~\ref{gen poly conj} fail in general when the relevant parameters are rational. 

We consider the prospects of quantitative bounds on the rate of convergence which do not depend on the value taken by the relevant parameters. We shall show that such a quantitative bound exists for Theorem~\ref{Beatty}, but not for Theorem~\ref{main} or Conjecture~\ref{gen poly conj}. 

The quantitative bound for Theorem~\ref{Beatty} (given by Theorem~\ref{Beatty quant}) will be derived by combining the Fourier techniques developed in Section~\ref{type A and C} and careful estimates of the characteristic function of $\omega(\floor{\alpha n + \beta})$ via its binomial moments. 

The lack of quantitative bounds for Theorem~\ref{main} or Conjecture~\ref{gen poly conj} heuristically trace back to the failure of qualitative convergence for rational values of the relevant parameters, and will follow from taking these parameters to be sufficiently Liouville-like. We show the lack of quantitative bounds formally for Conjecture~\ref{gen poly conj} (given by Theorem~\ref{gen poly quant}), a very similar proof yields the analogous result (with respect to the multivariate Kolmogorov distance) for Theorem~\ref{main}. 

\begin{df} \label{Kolmogorov dist}
    Given two real random variables $X, Y$ with cumulative distribution functions $F_X, F_Y$, the Kolmogorov distance between $X$ and $Y$ is given by 
    \begin{equation*}
        d_K(X, Y) := \|F_X - F_Y\|_\infty = \sup_{x\in\mathbb{R}}|F_X(x) - F_Y(x)| = \sup_{x\in\mathbb{R}}|\PP{X\leq x} - \PP{Y\leq x}|. 
    \end{equation*}
\end{df}

In the case of the original Erd\H{o}s--Kac theorem, the rate of convergence was settled by R\'enyi and Tur\'an~\cite{E-K-rate-of-conv}. 

\begin{thm}[R\'enyi and Tur\'an~\cite{E-K-rate-of-conv}, Theorem 3] \label{E-K quant}
    The Kolmogorov distance between the random variable 
    \begin{equation*}
        \frac{\omega(n) - \log\log N}{\sqrt{\log\log N}}
    \end{equation*}
    and the standard Gaussian $\mathcal{N}(0, 1)$ is bounded above by $O\left(\frac{1}{\sqrt{\log\log N}}\right)$ as $N\rightarrow\infty$. 
\end{thm}

Theorem~\ref{E-K quant} is tight as $\frac{\omega(n) - \log\log N}{\sqrt{\log\log N}}$ is supported on $\frac{1}{\sqrt{\log\log N}}\mathbb{Z}$, and matches the Berry-Esseen bound for sums of independent random variables. R\'enyi and Tur\'an's proof relies on complex-analytic methods leveraging the additivity of $\omega(n)$, which is often not available in generalisations of the Erd\H{o}s--Kac theorem. 

\subsection{Proof of Theorem~\ref{Beatty quant}}

We first establish Theorem~\ref{Beatty quant} on a quantitative rate of convergence for $\omega(\floor{\alpha n + \beta})$, which we restate here for clarity. As remarked above, we are unable to produce the tight $O\left(\frac{1}{\sqrt{\log\log N}}\right)$ bound in this setting; our bound will be weaker by a factor of $\log \log \log N$. 

\textbf{Theorem~\ref{Beatty quant}} \textit{Let $\alpha > 0$, $\beta\in\mathbb{R}$. For $n\sim U[N]$, the Kolmogorov distance between the random variable 
\begin{equation*}
    \frac{\omega(\floor{\alpha n + \beta}) - \log\log N}{\sqrt{\log\log N}}
\end{equation*}
and the standard Gaussian $\mathcal{N}(0, 1)$ is bounded above by $O_{\alpha, \beta}\left(\frac{\log\log\log N}{\sqrt{\log\log N}}\right)$ as $N\rightarrow\infty$, where the implied constant may depend on $\alpha$ and $\beta$. }

Throughout the proof, we consider $\alpha$ and $\beta$ to be absolute constants and implicitly assume that $N$ is sufficiently large in an absolute sense. We assume, without loss of generality\footnote{We may ensure this, for example, by considering the random variable
\begin{equation*}
    \frac{\omega(\floor{\alpha n + \beta}) - \log\log N}{\sqrt{\log\log N}}
\end{equation*}
for $n\sim U[c + 1, N + c]$ instead, where $c$ is an absolute constant taken so that $\alpha c + \beta\geq 1$. The Kolmogorov distance between this new random variable and the original is bounded by $O(N^{-1})$.}, that $\beta\geq 1$. In fact, the implied constant in Theorem~\ref{Beatty quant} may be taken uniformly whenever $\alpha$ is bounded away from $0$ and $\infty$, and $\beta/\alpha$ is bounded away from $-\infty$. The latter assumption ensures that the number of negative initial terms in the Beatty sequence $\floor{\alpha n + \beta}$ is bounded. 

We shall employ a classical lemma of Esseen~\cite{Esseen}, which establishes upper bounds on the Kolmogorov distance to the standard Gaussian via characteristic function estimates. 

\begin{lem}[Esseen's smoothing lemma~\cite{Esseen}] \label{esseen smoothing lem}
    Given a real random variable $X$, for any real number $A > 0$, 
    \begin{equation*}
        d_K(X, Z) \ll A^{-1} + \int_{-A}^A \left|\frac{\EE{e(tX)} - e^{-2\pi^2t^2}}{t}\right|dt, 
    \end{equation*}
    where $Z\sim \mathcal{N}(0, 1)$. 
\end{lem}

We employ a simplified version of the machinery used to prove Theorem~\ref{main}. In particular, we make use of the parameters $R, J$ and $\epsilon = J^{-1/2}$. However, contrary to the values given in Definition~\ref{params} for the proof of Theorem~\ref{main}, we now take $R := N^{1/\log\log N}$ and $J := N^{1/25}$. 

The analogue of Lemma~\ref{key lin alg} in this setting takes a very simple form. We reuse the notation $\mathcal{B}_N$ and $\mathcal{P}_N$ to denote analogous but newly defined sets of primes, given below. 

\begin{df}
    Let $\gamma = \gamma_N\in \mathcal{Q}(N^{1/8})$ minimise $|\gamma - \alpha|$. Let $\gamma = a/b$ for coprime positive integers $a, b$. Similar to the treatment in Section~\ref{lin alg}, let $\mathcal{B}_N$ denote the set of prime factors of $a$, together with the element 2. 
\end{df}

Since $\gamma\rightarrow \alpha$ as $N\rightarrow\infty$, we have $a\asymp b$. Note that any other $\gamma'\neq \gamma\in \mathcal{Q}(N^{1/8})$ satisfies
\begin{equation*}
    |\gamma' - \alpha|\gg N^{-1/4}. 
\end{equation*}

\begin{df}
    Let $\mathcal{P}_N$ denote the set of primes $p\leq R := N^{1/ \log \log N}$ not in $\mathcal{B}_N$. Let 
    \begin{equation*}
        \omega''_N(n) := \sum_{p\in\mathcal{P}_N} 1_{p|n}
    \end{equation*}
    denote the number of prime factors of $n$ in $\mathcal{P}_N$. In line with the Kubilius model, let 
    \begin{equation*}
        \omega'''_N := \sum_{p\in\mathcal{P}_N}X_p, 
    \end{equation*}
    where $X_p\sim \operatorname{Ber}(p^{-1})$ are independent random variables. Finally, let 
    \begin{equation*}
        s := \sum_{p\in\mathcal{P}_N} p^{-1}. 
    \end{equation*}
    denote the expectation of $\omega'''_N$. 
\end{df}

We first show that $\omega''_N(\floor{\alpha n + \beta})$ forms a good approximant of $\omega(\floor{\alpha n + \beta})$, before showing that the characteristic function of $\omega''_N(\floor{\alpha n + \beta})$ is well-approximated by the characteristic function of $\omega'''_N$. 

\begin{lem} \label{inverse sum estimates}
    Let $\mathcal{P}_N^c := \{p|p\notin \mathcal{P}_N, p\leq \alpha N + \beta\} = \mathcal{B}_N\cup \{p|R < p\leq \alpha N + \beta\}$ denote the set of primes which contribute to $\omega$ but not to $\omega''_N$. We have
    \begin{align*}
        &\sum_{p\in \mathcal{P}_N^c}p^{-1}\ll \log\log\log N, \\
        s = &\sum_{p\in \mathcal{P}_N}p^{-1} = \log\log N + O(\log\log\log N). 
    \end{align*}
\end{lem}

\begin{proof}
    It suffices to prove the first estimate, from which the second estimate follows via Mertens' estimate. Noting that $|\mathcal{B}_N|\ll 1 + \log a\leq \log N$, we have
    \begin{equation*}
        \sum_{p\in \mathcal{B}_N}p^{-1}\leq \sum_{\text{the first }|\mathcal{B}_N|\text{ primes } p}p^{-1}\ll \log \log \log N, 
    \end{equation*}
    by Mertens' estimate. Mertens' estimate also gives
    \begin{equation*}
        \sum_{R < p\leq \alpha N + \beta}p^{-1}\leq \log \log (\alpha N + \beta) - \log \log R + O(1) \ll \log \log \log N. 
    \end{equation*}
    Combining the two preceding inequalities, we arrive at the desired estimate for $\sum_{p\in \mathcal{P}_N^c}p^{-1}$. 
\end{proof}

\begin{lem} \label{quant esti lem triv}
    For a positive integer $d$, we have
    \begin{equation*}
        \mathbb{P}_{n\sim U[N]}(d\mid \floor{\alpha n + \beta}) \ll d^{-1}. 
    \end{equation*}
\end{lem}

\begin{proof}
    We simply note that $\floor{\alpha n + \beta}$ takes values in the interval $[\alpha N + \beta]$ of integers of length $O(N)$, in which there are at most $O(N/d)$ multiples of $d$, and each value of $\floor{\alpha n + \beta}$ is taken by at most $O(1)$ choices of $n$. 
\end{proof}

\begin{lem} \label{omega vs omega''}
    For some absolute constant $C$, 
    \begin{equation*}
        \mathbb{P}_{n\sim U[N]}\left(\omega(\floor{\alpha n + \beta}) - \omega''_N(\floor{\alpha n + \beta})\geq C\log\log\log N\right)\ll \frac{1}{\log\log N}. 
    \end{equation*}
\end{lem}

\begin{proof}
    Recall that $\mathcal{P}_N^c = \{p|p\notin \mathcal{P}_N, p\leq \alpha N + \beta\}$ denotes the set of primes which contribute to $\omega$ but not to $\omega''_N$. We have
    \begin{equation*}
         \omega(\floor{\alpha n + \beta}) - \omega''_N(\floor{\alpha n + \beta}) = \sum_{p\in \mathcal{P}_N^c} 1_{p\mid \floor{\alpha n + \beta}},
    \end{equation*}
    and for any non-negative integer $\ell$, 
    \begin{equation*}
        \binom{\omega(\floor{\alpha n + \beta}) - \omega''_N(\floor{\alpha n + \beta})}{\ell} = \sum_{p_1 < \dots < p_\ell\in \mathcal{P}_N^c} 1_{p_1\cdots p_\ell\mid \floor{\alpha n + \beta}},
    \end{equation*}
    By Lemma~\ref{quant esti lem triv}, we have
    \begin{align*}
        \mathbb{E}_{n\sim U[N]}\binom{\omega(\floor{\alpha n + \beta}) - \omega''_N(\floor{\alpha n + \beta})}{\ell} &= \sum_{p_1 < \dots < p_\ell\in \mathcal{P}_N^c} \mathbb{P}_{n\sim U[N]}\left(p_1\cdots p_\ell\mid \floor{\alpha n + \beta}\right)\\
        &\ll \sum_{p_1 < \dots < p_\ell\in \mathcal{P}_N^c} \left(p_1\cdots p_\ell\right)^{-1}\\
        &\leq \frac{1}{\ell!}\left(\sum_{p\in \mathcal{P}_N^c}p^{-1}\right)^\ell. 
    \end{align*}
    By Lemma~\ref{inverse sum estimates}, 
    \begin{equation*}
        \sum_{p\in\mathcal{P}_N^c}p^{-1}\leq C' \log \log \log N, 
    \end{equation*}
    for some absolute constant $C'$. Hence, we have
    \begin{equation*}
        \mathbb{E}_{n\sim U[N]}\binom{\omega(\floor{\alpha n + \beta}) - \omega''_N(\floor{\alpha n + \beta})}{\ell}\ll \frac{(C'\log\log\log N)^\ell}{\ell!}. 
    \end{equation*}
    By Markov's inequality, this implies
    \begin{equation*}
        \mathbb{P}_{n\sim U[N]}\left(\omega(\floor{\alpha n + \beta}) - \omega''_N(\floor{\alpha n + \beta})\geq 2\ell\right)\ll \frac{(C'\log\log\log N)^\ell}{\ell^\ell}. 
    \end{equation*}
    Taking $\ell = \max(2C', 10)\log\log\log N$ yields the desired result with $C = 2\max(2C', 10)$. 
\end{proof}

In essence, Lemma~\ref{omega vs omega''} reduces Theorem~\ref{Beatty quant} to the analogous statement with the truncated object $\omega''_N$ in place of $\omega$. We now show that the characteristic function of $\omega''_N$ is well approximated by that of $\omega'''_N$. 

\begin{lem} \label{quant esti lem}
    For any positive integer $d$ coprime to $a$, we have
    \begin{equation*}
        \left|\mathbb{P}_{n\sim U[N]}(d\mid \floor{\alpha n + \beta}) - d^{-1}\right| \ll N^{-1/200}. 
    \end{equation*}
\end{lem}

\begin{proof}
    We may assume that $d\leq N^{1/200}$, otherwise the desired claim follows from Lemma~\ref{quant esti lem triv}. 
    
    We make use of the approximation technique developed in Subsection~\ref{indicator approximation 2}. Note that the analysis in Subsection~\ref{indicator approximation 2} did not depend on the primality of $p$, and remains valid if we take $\epsilon = N^{-1/50}$. Applying analogues of Lemma~\ref{E and E''} and Corollary~\ref{E'' cor} with $d$ in place of $p$, it suffices to show that
    \begin{equation} \label{quant lem sum}
        \sum_{m\in\frac{1}{d}\mathbb{Z}\backslash\mathbb{Z}}\frac{\min(1, m^{-1})\min(1, \epsilon^{-1}m^{-1})}{d\max(N^{1/2}\|\alpha m\|_\mathbb{Z}, 1)}\ll N^{-1/200}
    \end{equation}
    Let $J = N^{1/25} = \epsilon^{-2}$, the contribution of $|m|\geq J$ is bounded above by $O(\epsilon^{-1}J^{-1}) = O(N^{-1/50})$. 
    
    If $|m| < J$, let $m'$ denote a closest integer to $\alpha m$. We have $d|m|, d|m'|\leq N^{1/8}$. Hence $m'/m = (dm')/(dm)\in \mathcal{Q}(N^{1/8})$. Since $d\mid dm'$ and $d\nmid dm$, it follows that $m'/m = (dm')/(dm)\neq \gamma$. Hence
    \begin{equation*}
        \|\alpha m\|_\mathbb{Z} = |\alpha m - m'| = |m||m'/m - \alpha|\gg N^{-1/4}/d\gg N^{-51/200}. 
    \end{equation*}
    Hence the contribution from $|m| < J$ to the sum (\ref{quant lem sum}) is bounded above by 
    \begin{equation*}
        O\left(\sum_{m\in\frac{1}{d}\mathbb{Z}\backslash\mathbb{Z}, |m| < J}\frac{\min(1, m^{-1})}{dN^{49/200}}\right)\ll 
        \log(J)N^{-49/200} \ll \log(N)N^{-49/200}\ll N^{-1/200}. \qedhere
    \end{equation*}
\end{proof}

We note that
\begin{align*}
    \binom{\omega''(\floor{\alpha n + \beta})}{\ell} &= \sum_{p_1 < \dots < p_\ell\in \mathcal{P}_N} 1_{p_1\cdots p_\ell\mid \floor{\alpha n + \beta}}, \\
    \binom{\omega'''}{\ell} &= \sum_{p_1 < \dots < p_\ell\in \mathcal{P}_N} X_{p_1}\cdots X_{p_\ell}, 
\end{align*}
and hence
\begin{align*}
    \mathbb{E}_{n\sim U[N]}\binom{\omega''(\floor{\alpha n + \beta})}{\ell} &= \sum_{p_1 < \dots < p_\ell\in \mathcal{P}_N} \mathbb{P}_{n\sim U[N]}\left(p_1\cdots p_\ell\mid \floor{\alpha n + \beta}\right), \\
    \mathbb{E}\binom{\omega'''}{\ell} &= \sum_{p_1 < \dots < p_\ell\in \mathcal{P}_N} (p_1\cdots p_\ell)^{-1}. 
\end{align*}
As such, we have the following corollaries of Lemmas~\ref{quant esti lem triv} and~\ref{quant esti lem}, respectively. 

\begin{cor} \label{binom approx tail}
    For any non-negative integer $\ell$, 
    \begin{equation*}
        \left|\mathbb{E}_{n\sim U[N]}\binom{\omega_N''(\floor{\alpha n + \beta})}{\ell} - \mathbb{E}\binom{\omega_N'''}{\ell}\right|\ll\frac{(\log\log N)^\ell}{\ell!}. 
    \end{equation*}
\end{cor}

\begin{proof}
    By Lemma~\ref{quant esti lem triv} and Mertens' estimate, we have
    \begin{align*}
        \left|\mathbb{E}_{n\sim U[N]}\binom{\omega_N''(\floor{\alpha n + \beta})}{\ell} - \mathbb{E}\binom{\omega_N'''}{\ell}\right|&\ll\sum_{p_1 < \dots < p_\ell\in \mathcal{P}_N} (p_1\cdots p_\ell)^{-1}\\
        &\leq \frac{\left(\sum_{p\in \mathcal{P}_N} p^{-1}\right)^\ell}{\ell!}\\
        &\ll\frac{(\log\log N)^\ell}{\ell!}. \qedhere
    \end{align*}
\end{proof}

\begin{cor} \label{binom approx}
    For any non-negative integer $\ell$, 
    \begin{equation*}
        \left|\mathbb{E}_{n\sim U[N]}\binom{\omega_N''(\floor{\alpha n + \beta})}{\ell} - \mathbb{E}\binom{\omega_N'''}{\ell}\right|\ll N^{-1/200}\frac{N^{\ell/\log\log N}}{\ell!}. 
    \end{equation*}
\end{cor}

\begin{proof}
    By Lemma~\ref{quant esti lem}, we have
    \begin{align*}
        \left|\mathbb{E}_{n\sim U[N]}\binom{\omega_N''(\floor{\alpha n + \beta})}{\ell} - \mathbb{E}\binom{\omega_N'''}{\ell}\right|&\ll \sum_{p_1 < \dots < p_\ell\in \mathcal{P}_N}N^{-1/200}\\
        &\leq N^{-1/200}\binom{R}{\ell}\\
        &\leq N^{-1/200}\frac{N^{\ell/\log\log N}}{\ell!}. \qedhere
    \end{align*}
\end{proof}

To transform Corollaries~\ref{binom approx tail} and~\ref{binom approx} into a statement on characteristic functions, we make use of the identity
\begin{equation} \label{binom to char}
    e(tx) = \sum_{\ell \geq 0}(e(t) - 1)^\ell \binom{x}{\ell}, 
\end{equation}
for any real number $t$ and any non-negative integer $x$, which follows from the binomial theorem. 

\begin{lem} \label{omega'' and omega'''}
    For any real number $t$ with $|t|\leq 10^{-5}$, we have
    \begin{equation*}
        \left|\mathbb{E}_{n\sim U[N]}(e(t\omega''_N(\floor{\alpha n + \beta}))) - \mathbb{E}(e(t\omega'''_N))\right|\ll (\log N)^{-1/400}|t|.
    \end{equation*}
\end{lem}

\begin{proof}
    By (\ref{binom to char}), we have 
    \begin{equation*}
        \mathbb{E}_{n\sim U[N]}(e(t\omega''_N(\floor{\alpha n + \beta}))) - \mathbb{E}(e(t\omega'''_N)) = \sum_{\ell\geq 0}(e(t) - 1)^\ell\left(\mathbb{E}_{n\sim U[N]}\binom{\omega_N''(\floor{\alpha n + \beta})}{\ell} - \mathbb{E}\binom{\omega_N'''}{\ell}\right). 
    \end{equation*}
    Since the $\ell = 0$ term vanishes, we may in fact assume that $\ell \geq 1$. By Corollaries~\ref{binom approx tail} and~\ref{binom approx}, we have
    \begin{align*}
        &\left|\mathbb{E}_{n\sim U[N]}(e(t\omega''_N(\floor{\alpha n + \beta}))) - \mathbb{E}(e(t\omega'''_N))\right|\\
        &\ll \sum_{\ell\geq 1}|e(t) - 1|^\ell \min\left(N^{-1/200}\frac{N^{\ell/\log\log N}}{\ell!}, \frac{(\log\log N)^\ell}{\ell!}\right)\\
        &\leq \sum_{1\leq \ell\leq \frac{\log\log N}{400}}|e(t) - 1|^\ell N^{-1/200}\frac{N^{\ell/\log\log N}}{\ell!} + \sum_{\ell > \frac{\log\log N}{400}}|e(t) - 1|^\ell \frac{(\log\log N)^\ell}{\ell!}\\
        &\leq \sum_{1\leq \ell\leq \frac{\log\log N}{400}}|e(t) - 1|^\ell \frac{N^{-1/400}}{\ell!} + \sum_{\ell > \frac{\log\log N}{400}}|e(t) - 1|^\ell \frac{(\log\log N)^\ell}{\ell!}\\
        &\ll N^{-1/400}(\exp(|e(t) - 1|) - 1) + \sum_{\ell > \frac{\log\log N}{400}}\left(|e(t) - 1|\frac{e\log\log N}{\ell}\right)^\ell\\
        &\ll N^{-1/400}|t| + \sum_{\ell > \frac{\log\log N}{400}}\left(\frac{2\pi e|t|}{1/400}\right)^\ell\\
        &\ll N^{-1/400}|t| + \left(10^4|t|\right)^{\frac{\log\log N}{400}}\\
        &\ll (\log N)^{-1/400}|t|. \qedhere
    \end{align*}
\end{proof}

As 
\begin{equation*}
    \omega'''_N = \sum_{p\in\mathcal{P}_N}X_p
\end{equation*}
is a sum of independent random variables, the characteristic function $\omega'''_N$ may be given as
\begin{equation*}
    \mathbb{E}(e(t\omega'''_N)) = \prod_{p\in\mathcal{P}_N}\left(1 + \frac{e(t) - 1}{p}\right). 
\end{equation*}
Recall that $s = \sum_{p\in\mathcal{P}_N}p^{-1}$. Taking the logarithm, we have
\begin{equation} \label{log omega''' char}
    \log \mathbb{E}(e(t\omega'''_N)) = \sum_{p\in\mathcal{P}_N}\left(\frac{e(t) - 1}{p} + O\left(\frac{t^2}{p^2}\right)\right) = s(e(t) - 1) + O(t^2). 
\end{equation}

We now compare the characteristic functions of $\frac{\omega'''_N - s}{\sqrt{s}}$ and the standard Gaussian distribution. We directly bound their difference over $[-s^{-1/6}, s^{-1/6}]$ (Lemma~\ref{omega''' char estimate}) and provide tail estimates over $[-\sqrt{s}/2, \sqrt{s}/2]$ (Lemma~\ref{omega''' char tail estimate}). 

\begin{lem} \label{omega''' char estimate}
    For $|t|\leq s^{1/6}$, 
    \begin{equation*}
        \left|\EE{e\left(t\frac{\omega'''_N - s}{\sqrt{s}}\right)} - e^{-2\pi^2t^2}\right|\ll \frac{e^{-2\pi^2t^2}t^2(|t| + 1)}{\sqrt{s}}. 
    \end{equation*}
\end{lem}

\begin{proof}
    We make use of the expansion $e(x) = 1 + 2\pi i x - 2\pi^2x^2 + O(|x|^3)$ for $x\in [-1, 1]$. We apply this expansion to $x = t/\sqrt{s}$. By (\ref{log omega''' char}), we have
    \begin{align*}
        \log \EE{e\left(t\frac{\omega'''_N - s}{\sqrt{s}}\right)} &= \left(e\left(\frac{t}{\sqrt{s}}\right) - 1\right)s + \log e\left(t\frac{- s}{\sqrt{s}}\right) + O(t^2/s)\\
        &= \left(e\left(\frac{t}{\sqrt{s}}\right) - 1 - 2\pi i\frac{t}{\sqrt{s}}\right)s + O(t^2/s)\\
        &= \left(-2\pi^2 t^2/s + O(|t|^3/s^{3/2})\right)s + O(t^2/s)\\
        &= -2\pi^2t^2 + O(t^2(|t| + 1)/\sqrt{s}). 
    \end{align*}
    Note that the error term $O(t^2(|t| + 1)/\sqrt{s})\ll 1$ for $t\leq s^{1/6}$. Therefore, 
    \begin{equation*}
        \EE{e\left(t\frac{\omega'''_N - s}{\sqrt{s}}\right)} = e^{-2\pi^2t^2}(1 + O(t^2(|t| + 1)/\sqrt{s})),
    \end{equation*}
    giving the desired estimate. 
\end{proof}

\begin{lem} \label{omega''' char tail estimate}
    For $|t|\leq \sqrt{s}/2$, 
    \begin{equation*}
        \left|\EE{e\left(t\frac{\omega'''_N - s}{\sqrt{s}}\right)}\right|\leq e^{-4t^2}. 
    \end{equation*}
\end{lem}

\begin{proof}
    We note that whenever $|t|\leq 1/2$ and $p\geq 2$, 
    \begin{equation*}
        \left|1 + \frac{e(t) - 1}{p}\right| = \sqrt{1 + 2(1 - p^{-1})p^{-1}(\cos(2\pi t) - 1)}\leq \sqrt{1 - 8p^{-1}t^2}
        \leq \exp(-4p^{-1}t^2). 
    \end{equation*}
    Therefore, when $|t|\leq \sqrt{s}/2$, 
    \begin{equation*}
        \left|\EE{e\left(t\frac{\omega'''_N - s}{\sqrt{s}}\right)}\right| = \left|\EE{e\left(\frac{t}{\sqrt{s}}\omega'''_N\right)}\right| = \prod_{p\in\mathcal{P}_N}\left|1 + \frac{e(t/\sqrt{s}) - 1}{p}\right|\leq \exp(-4s(t/\sqrt{s})^2)\leq e^{-4t^2}. \qedhere
    \end{equation*}
\end{proof}

We now combine Lemmas~\ref{omega''' char estimate} and~\ref{omega''' char tail estimate} to estimate the integral appearing in Esseen's smoothing lemma. 

\begin{cor} \label{esseen omega''' main term}
    \begin{equation*}
        \int_{-\sqrt{s}/2}^{\sqrt{s}/2}\left|\frac{\mathbb{E}\left(e\left(t\frac{\omega'''_N - s}{\sqrt{s}}\right)\right) - e^{-2\pi^2t^2}}{t}\right|dt \ll \frac{1}{\sqrt{\log\log N}}. 
    \end{equation*}
\end{cor}

\begin{proof}
    We split the integral into pieces for $|t|\leq s^{1/6}$ and $|t| > s^{1/6}$, to which we apply Lemmas~\ref{omega''' char estimate} and~\ref{omega''' char tail estimate}, respectively. We have
    \begin{align*}
        &\int_{-\sqrt{s}/2}^{\sqrt{s}/2}\left|\frac{\mathbb{E}\left(e\left(t\frac{\omega'''_N - s}{\sqrt{s}}\right)\right) - e^{-2\pi^2t^2}}{t}\right|dt\\
        &= \int_{|t|\leq s^{1/6}}\left|\frac{\mathbb{E}\left(e\left(t\frac{\omega'''_N - s}{\sqrt{s}}\right)\right) - e^{-2\pi^2t^2}}{t}\right|dt + \int_{s^{1/6} < |t|\leq \sqrt{s}/2}\left|\frac{\mathbb{E}\left(e\left(t\frac{\omega'''_N - s}{\sqrt{s}}\right)\right) - e^{-2\pi^2t^2}}{t}\right|dt\\
        &\ll \int_{|t|\leq s^{1/6}}\frac{e^{-2\pi^2t^2}|t|(|t| + 1)}{\sqrt{s}}dt + \int_{s^{1/6} < |t|\leq \sqrt{s}/2}e^{-4t^2}dt\\
        &\ll \frac{1}{\sqrt{s}} + e^{-4s^{1/3}}\\
        &\ll \frac{1}{\sqrt{\log\log N}}. \qedhere
    \end{align*}
\end{proof}

The corresponding result for $\omega''_N(\floor{\alpha n + \beta})$ in place of $\omega'''_N$ now follows from Lemma~\ref{omega'' and omega'''}. 

\begin{cor} \label{esseen main term}
    \begin{equation*}
        \int_{-10^{-5}\sqrt{s}}^{10^{-5}\sqrt{s}}\left|\frac{\mathbb{E}_{n\sim U[N]}\left(e\left(t\frac{\omega''_N(\floor{\alpha n + \beta}) - s}{\sqrt{s}}\right)\right) - e^{-2\pi^2t^2}}{t}\right|dt \ll \frac{1}{\sqrt{\log\log N}}. 
    \end{equation*}
\end{cor}

\begin{proof}
    By Lemmas~\ref{omega'' and omega'''}, \ref{esseen omega''' main term} and the triangle inequality, we have
    \begin{align*}
        &\int_{-10^{-5}\sqrt{s}}^{10^{-5}\sqrt{s}}\left|\frac{\mathbb{E}_{n\sim U[N]}\left(e\left(t\frac{\omega''_N(\floor{\alpha n + \beta}) - s}{\sqrt{s}}\right)\right) - e^{-2\pi^2t^2}}{t}\right|dt\\
        &\leq \int_{-10^{-5}\sqrt{s}}^{10^{-5}\sqrt{s}}\left|\frac{\mathbb{E}\left(e\left(t\frac{\omega'''_N - s}{\sqrt{s}}\right)\right) - e^{-2\pi^2t^2}}{t}\right|dt\\
        &\qquad+ \int_{-10^{-5}\sqrt{s}}^{10^{-5}\sqrt{s}}\left|\frac{\mathbb{E}_{n\sim U[N]}\left(e\left(t\frac{\omega''_N(\floor{\alpha n + \beta}) - s}{\sqrt{s}}\right)\right) - \mathbb{E}\left(e\left(t\frac{\omega'''_N - s}{\sqrt{s}}\right)\right)}{t}\right|dt\\
        &\ll \frac{1}{\sqrt{\log\log N}} + \int_{-10^{-5}\sqrt{s}}^{10^{-5}\sqrt{s}}\left|\frac{\mathbb{E}_{n\sim U[N]}\left(e\left(\frac{t}{\sqrt{s}}\omega''_N(\floor{\alpha n + \beta})\right)\right) - \mathbb{E}\left(e\left(\frac{t}{\sqrt{s}}\omega'''_N\right)\right)}{t}\right|dt\\
        &\ll \frac{1}{\sqrt{\log\log N}} + \int_{-10^{-5}\sqrt{s}}^{10^{-5}\sqrt{s}}\frac{(\log N)^{-1/400}|t/\sqrt{s}|}{|t|}dt\\
        &\ll \frac{1}{\sqrt{\log\log N}}. \qedhere
    \end{align*}
\end{proof}

\begin{cor} \label{quant s cor}
    \begin{equation*}
        d_K\left(\frac{\omega''_N(\floor{\alpha n + \beta}) - s}{\sqrt{s}}, Z\right)\ll \frac{1}{\sqrt{\log\log N}}, 
    \end{equation*}
    where $Z\sim \mathcal{N}(0, 1)$. 
\end{cor}

\begin{proof}
    This follows from Corollary~\ref{esseen main term} by Esseen's smoothing lemma (Lemma~\ref{esseen smoothing lem}) with $A = 10^{-5}\sqrt{s}\asymp \sqrt{\log\log N}$. 
\end{proof}

\begin{cor} \label{omega'' quant}
    \begin{equation*}
        d_K\left(\frac{\omega''_N(\floor{\alpha n + \beta}) - \log\log N}{\sqrt{\log \log N}}, Z\right)\ll \frac{\log\log\log N}{\sqrt{\log\log N}}, 
    \end{equation*}
    where $Z\sim \mathcal{N}(0, 1)$. 
\end{cor}

\begin{proof}
    By Corollary~\ref{quant s cor}, 
    \begin{equation*}
        d_K\left(\frac{\omega''_N(\floor{\alpha n + \beta}) - \log\log N}{\sqrt{\log \log N}}, Z'\right)\ll \frac{1}{\sqrt{\log\log N}}, 
    \end{equation*}
    where $Z'\sim \mathcal{N}\left(\frac{s - \log\log N}{\sqrt{\log\log N}}, \frac{s}{\log\log N}\right)$. However, since $s = \log\log N + O(\log \log \log N)$ by Lemma~\ref{inverse sum estimates}, we have
    \begin{equation*}
        d_K(Z, Z')\ll \frac{\log\log\log N}{\sqrt{\log\log N}}. 
    \end{equation*}
    The desired result follows from the triangle inequality for the Kolmogorov distance. 
\end{proof}

We are now ready to combine Lemma~\ref{omega vs omega''} and Corollary~\ref{omega'' quant} to show Theorem~\ref{Beatty quant}. 

\begin{proof}[Proof of Theorem~\ref{Beatty quant}]
    Let $Z\sim \mathcal{N}(0, 1)$. For any $x\in \mathbb{R}$, since $\omega(\floor{\alpha n + \beta})\geq\omega''_N(\floor{\alpha n + \beta})$, we have, by Corollary~\ref{omega'' quant}, 
    \begin{align*}
        &\mathbb{P}_{n\sim U[N]}\left(\frac{\omega(\floor{\alpha n + \beta}) - \log\log N}{\sqrt{\log\log N}}\leq x\right)\\
        &\leq \mathbb{P}_{n\sim U[N]}\left(\frac{\omega''_N(\floor{\alpha n + \beta}) - \log\log N}{\sqrt{\log\log N}}\leq x\right)\\
        & = \mathbb{P}(Z\leq x) + O\left(\frac{\log\log\log N}{\sqrt{\log\log N}}\right). 
    \end{align*}
    On the other hand, by Lemma~\ref{omega vs omega''} and Corollary~\ref{omega'' quant}, we have
    \begin{align*}
        &\mathbb{P}_{n\sim U[N]}\left(\frac{\omega(\floor{\alpha n + \beta}) - \log\log N}{\sqrt{\log\log N}}\leq x\right)\\
        &\geq \mathbb{P}_{n\sim U[N]}\left(\frac{\omega''_N(\floor{\alpha n + \beta}) - \log\log N}{\sqrt{\log\log N}}\leq x - \frac{C\log\log\log N}{\sqrt{\log\log N}}\right)\\
        &\qquad - \mathbb{P}_{n\sim U[N]}\left(\omega(\floor{\alpha n + \beta}) - \omega''_N(\floor{\alpha n + \beta})\geq C\log\log\log N\right)\\
        & = \mathbb{P}\left(Z\leq x - \frac{C\log\log\log N}{\sqrt{\log\log N}}\right) + O\left(\frac{\log\log\log N}{\sqrt{\log\log N}}\right)\\
        & = \mathbb{P}\left(Z\leq x\right) + O\left(\frac{\log\log\log N}{\sqrt{\log\log N}}\right). 
    \end{align*}
    Therefore, for any $x\in \mathbb{R}$, 
    \begin{equation*}
        \left|\mathbb{P}_{n\sim U[N]}\left(\frac{\omega(\floor{\alpha n + \beta}) - \log\log N}{\sqrt{\log\log N}}\leq x\right) - \PP{Z\leq x}\right|\ll \frac{\log\log\log N}{\sqrt{\log\log N}}. 
    \end{equation*}
    Hence
    \begin{equation*}
        d_K\left(\frac{\omega(\floor{\alpha n + \beta}) - \log\log N}{\sqrt{\log\log N}}, Z\right)\ll \frac{\log\log\log N}{\sqrt{\log\log N}}. \qedhere
    \end{equation*}
\end{proof}

\subsection{Proof of Theorem~\ref{gen poly quant}}

We now turn to Theorem~\ref{gen poly quant} on the impossibility of universal quantitative bounds on the rate of convergence in the case of higher-degree generalised polynomials. For clarity, we restate Theorem~\ref{gen poly quant} here. 

\textbf{Theorem~\ref{gen poly quant}} \textit{For any sequence $\eta_N\searrow 0$ as $N\rightarrow\infty$, there exists a polynomial $f$ with at least one irrational non-constant coefficient and a positive leading coefficient, such that the following holds. For $n\sim U[N]$, the Kolmogorov distance between the random variable 
\begin{equation*}
    \frac{\omega(\floor{f(n)}) - \log\log N}{\sqrt{\log\log N}}
\end{equation*}
and the standard Gaussian $\mathcal{N}(0, 1)$ is not $O(\eta_N)$ as $N\rightarrow\infty$. }

As discussed earlier, we shall leverage the potential failure of qualitative convergence when the coefficients of $f$ are rational to establish the lack of quantitative bounds when its coefficients are sufficiently Liouville-like. 

\begin{proof}
    We shall in fact exhibit such a polynomial $f$ of any given degree $d\geq 2$. 
    
    We recursively construct an increasing sequence $10^{10}\leq a_1 < a_2 <\cdots$ of positive integers such that, letting $b_m = a_1\cdots a_m$ and $N_m = \floor{a_m^{1/d}/2}$, we have
    \begin{enumerate}
        \item $\log N_{m + 1}\geq (10b_m)^{10}$, 
        \item $\eta_{N_{m + 1}}^{-1}\geq m b_m$, 
        \item $\mathbb{P}_{n\sim U[N_{m + 1}/b_m]}(\omega(n)\leq 0.9\log\log N_{m + 1})\leq 0.1$, 
        \item $\mathbb{P}_{n\sim U[N_{m + 1}/b_m]}(\omega(b_mn - 1)\leq 0.9\log\log N_{m + 1})\leq 0.1$. 
    \end{enumerate}
    Note that all conditions are satisfied for sufficiently large $N_{m + 1}$ (\emph{i.e.} sufficiently large $a_{m + 1}$) with respect to $a_1, \dots, a_m$. For the third and fourth items, this follows either from suitable Erd\H{o}s--Kac laws or from the weaker Hardy-Ramanujan theorem~\cite{Hardy-Ramanujan}. Therefore, a valid choice of $a_{m + 1}$ may always be made. 

    Let $\alpha := \sum_i a_i^{-1}$, with rational approximants $\alpha_m := \sum_{i\leq m} a_i^{-1}$. Note that $\alpha_m\in b_m^{-1}\mathbb{Z}$ and $|\alpha - \alpha_m|\ll a_{m + 1}^{-1}\leq N_{m + 1}^{-1}\lll b_m^{-1}$. Therefore, $\alpha$ must be irrational. 
    
    Let $g(n) = (n - 1)n^{d - 1}$, and take $f(n) = \alpha g(n)$ to be a polynomial of degree $d\geq 2$ with positive irrational leading coefficient $\alpha$. 

    Since $\alpha_m \in b_m^{-1}\mathbb{Z}$, for any positive integer $n$, $\alpha_m g(b_mn)$ is an integer. For any positive integer $n \leq N_{m + 1}/b_m$, we have
    \begin{align*}
        f(b_mn) - \alpha_mg(b_mn) = (\alpha - \alpha_m)g(b_mn)\in [0, 2a_{m + 1}^{-1}N_{m + 1}^d)\subseteq [0, 1). 
    \end{align*}
    Therefore, 
    \begin{equation*}
        \floor{f(b_m n)} = \alpha_m g(b_m n) = \alpha_m (b_mn - 1)(b_mn)^{d - 1}. 
    \end{equation*}
    In particular, $\floor{f(b_m n)}$ is a multiple of $n(b_m n - 1)$. By the third and fourth items above, for $n\sim U[N_{m + 1}/b_m]$, we have
    \begin{equation*}
        \omega(\floor{f(b_m n)})\geq \omega(n) + \omega(b_mn - 1)\geq 1.8\log\log N_{m + 1}
    \end{equation*}
    with probability at least $0.8$. Therefore, 
    \begin{align*}
        &\mathbb{P}_{n\sim U[N_{m + 1}]}\left(\frac{\omega(\floor{f(n)}) - \log\log N_{m + 1}}{\sqrt{\log\log N_{m + 1}}}\geq 0.8\sqrt{\log\log N_{m + 1}}\right) \\
        &\geq \frac{\floor{N_{m + 1}/b_m}}{N_{m + 1}}\cdot\mathbb{P}_{n\sim U[N_{m + 1}/b_m]}\left(\frac{\omega(\floor{f(b_mn)}) - \log\log N_{m + 1}}{\sqrt{\log\log N_{m + 1}}}\geq 0.8\sqrt{\log\log N_{m + 1}}\right)\\
        &\geq (1 - o(1))b_m^{-1}\cdot 0.8. 
    \end{align*}
    On the other hand, for $Z\sim\mathcal{N}(0, 1)$, 
    \begin{equation*}
        \mathbb{P}(Z\geq 0.8\sqrt{\log\log N_{m + 1}})\leq (\log N_{m + 1})^{-0.32}\leq 0.1b_m^{-1}, 
    \end{equation*}
    since $\log N_{m + 1}\geq (10b_m)^{10}$. Hence, for $n\sim U[N_{m + 1}]$, 
    \begin{equation*}
        d_K\left(\frac{\omega(\floor{f(n)}) - \log\log N_{m + 1}}{\sqrt{\log\log N_{m + 1}}}, Z\right)\geq (0.7 - o(1))b_m^{-1}\geq (0.7 - o(1)) m\eta_{N_{m + 1}}\ggg \eta_{N_{m + 1}}, 
    \end{equation*}
    since $\eta_{N_{m + 1}}^{-1}\geq mb_m$. 
\end{proof}

\section{Concluding Remarks}

We end by discussing some conjectures on the quantitative convergence of Erd\H{o}--Kac laws. A fundamental lower bound on the Kolmogorov distance, on the order of $\frac{1}{\sqrt{\log\log N}}$, arises for a wide class of Erd\H{o}s--Kac laws since the random variable
\begin{equation*}
    \frac{\omega(\cdot) - \log\log N}{\sqrt{\log\log N}}
\end{equation*}
is supported on $\frac{1}{\sqrt{\log\log N}}\mathbb{Z}$. Using complex-analytic methods, R\'enyi and Tur\'an~\cite{E-K-rate-of-conv} established the matching upper bound (given as Theorem~\ref{E-K quant} here) in the case of the Erd\H{o}s--Kac theorem itself. Prior to their work, Kubilius~\cite{Kubilius} established a weaker $O\left(\frac{\log\log\log N}{\sqrt{\log\log N}}\right)$ upper bound via probabilistic methods. Harper~\cite{harper} offers a modern probabilistic approach recovering Kubilius' bound via Stein's method. At present, tight bounds on the rate of convergence of Erd\H{o}s--Kac laws still require the use of complex-analytic methods. 

The complex-analytic methods of R\'enyi and Tur\'an do not readily adapt to the case of the Erd\H{o}s--Kac law for Beatty sequences $\floor{\alpha n + \beta}$. The upper bound $O\left(\frac{\log\log\log N}{\sqrt{\log\log N}}\right)$ established by Theorem~\ref{Beatty quant} in this case reflects the use of probabilistic methods. It is natural to conjecture that the optimal $O\left(\frac{1}{\sqrt{\log\log N}}\right)$ upper bound on the rate of convergence may hold. 

\begin{conj}
    Let $\alpha > 0$, $\beta\in\mathbb{R}$. For $n\sim U[N]$, the Kolmogorov distance between the random variable 
    \begin{equation*}
        \frac{\omega(\floor{\alpha n + \beta}) - \log\log N}{\sqrt{\log\log N}}
    \end{equation*}
    and the standard Gaussian $\mathcal{N}(0, 1)$ is bounded above by $O_{\alpha, \beta}\left(\frac{1}{\sqrt{\log\log N}}\right)$ as $N\rightarrow\infty$, where the implied constant may depend on $\alpha$ and $\beta$. 
\end{conj}

In the case of Erd\H{o}s--Kac laws for higher-degree generalised polynomials, Theorem~\ref{gen poly quant} established the lack of quantitative convergence bounds in general. However, the method of Theorem~\ref{gen poly quant} may only produce examples of generalised polynomials $\floor{f(n)}$ with arbitrarily slow Erd\H{o}s--Kac convergence in which the polynomial $f$ is reducible over $\mathbb{R}$. As such, we conjecture the existence of quantitative bounds for the rate of convergence when $f$ is irreducible over $\mathbb{R}$. Since a non-constant polynomial irreducible over $\mathbb{R}$ must be linear or quadratic, and the linear Beatty case has been treated by Theorem~\ref{Beatty quant}, it suffices to consider irreducible quadratic polynomials. 

\begin{conj}
    There exists a sequence $\eta_N\searrow 0$ as $N\rightarrow\infty$, such that the following holds for any real numbers $\alpha > 0, \beta, \gamma$, where $\alpha$ and $\beta$ are not both rational and $\beta^2 < 4\alpha\gamma$. For $n\sim U[N]$, the Kolmogorov distance between the random variable 
    \begin{equation*}
        \frac{\omega(\floor{\alpha n^2 + \beta n + \gamma}) - \log\log N}{\sqrt{\log\log N}}
    \end{equation*}
    and the standard Gaussian $\mathcal{N}(0, 1)$ is $O_{\alpha, \beta, \gamma}(\eta_N)$ as $N\rightarrow\infty$, where the implied constant may depend on $\alpha, \beta$ and $\gamma$. 
\end{conj}

\textbf{AI usage declaration.} The mathematical content of this paper is human-generated. ChatGPT Pro is used to polish an initial version of this manuscript. 

\noindent \textbf{Acknowledgements.} The author is grateful to Joni Ter\"av\"ainen for guidance and helpful suggestions. 

\bibliographystyle{abbrv}
\bibliography{mybib}

\end{document}